\documentclass[11pt,a4paper]{amsart}
\usepackage{amssymb,amsmath,epsfig,graphics,mathrsfs}

\usepackage{graphicx}
\usepackage{cancel}
\usepackage{comment}

\usepackage{fancyhdr}
\usepackage{hyperref}
\hypersetup{
 colorlinks   = true,
 urlcolor     = blue,
 linkcolor    = blue,
 citecolor   = red ,
 bookmarksopen=true
}

\usepackage{amsmath}
\usepackage{amsfonts}
\usepackage{amssymb}
\usepackage{amsthm}
\usepackage{epsfig,graphics,mathrsfs}
\usepackage{graphicx}

\usepackage{latexsym} 
\usepackage{amssymb,amsmath}
\usepackage{amsthm}
\usepackage{longtable,booktabs,setspace} 
\usepackage{url}

\usepackage[usenames, dvipsnames]{color} 

\usepackage{hyperref}

\newcommand*\MSC[1][1991]{\par\leavevmode\hbox{%
\textit{\,\,\,\,\, #1 Mathematical subject classification:\ }}}
\newcommand\blfootnote[1]{%
  \begingroup
  \renewcommand\thefootnote{}\footnote{#1}%
  \addtocounter{footnote}{-1}%
  \endgroup
}

\def \de {\partial}

\def \N {\mathbb{N}}

\def \phi {\varphi}

\def \RN {\mathbb{R}^N}
\def \R {\mathbb{R}}

\def \K {\mathscr{K}}

\def \G{\Gamma}

\newcommand{\sA}{\mathscr A}

\newcommand{\Rn}{\mathbb R^n}
\newcommand{\Rm}{\mathbb R^m}

\newcommand{\Hn}{\mathbb H^n}

\newcommand{\p}{\partial}

\newcommand{\bG}{\mathbb {G}}
\newcommand{\bg}{\mathfrak g}

\newcommand{\la}{\lambda}

\numberwithin{equation}{section}

\newcommand{\beq}{\begin{equation}}
\newcommand{\bea}[1]{\begin{array}{#1} }
\newcommand{\eeq}{ \end{equation}}
\newcommand{\ea}{ \end{array}}

\newcommand{\ve}{\varepsilon}

\newcommand{\I}{\mathscr I_{HL}}

\DeclareMathOperator{\csch}{csch}
\DeclareMathOperator{\cotanh}{cotanh}

\newtheorem{theorem}{Theorem}[section]
\newtheorem{lemma}[theorem]{Lemma}
\newtheorem{proposition}[theorem]{Proposition}

\newtheorem{remark}[theorem]{Remark}

\newtheorem{example}[theorem]{Example}

\DeclareMathOperator{\sech}{sech}

\numberwithin{equation}{section}

\begin{document}

\title[Heat kernels for a class of hybrid etc.]{Heat kernels for a class of hybrid evolution equations}

{\blfootnote{\MSC[2020]{35K08, 35R03, 53C17}}}

\date{}

\begin{abstract}
The aim of this paper is to construct (explicit) heat kernels for some \emph{hybrid} evolution equations which arise in physics, conformal geometry and subelliptic PDEs. Hybrid means that the relevant partial differential operator appears in the form $\mathscr L_1 + \mathscr L_2 - \p_t$, but the variables cannot be decoupled. As a consequence, the relative heat kernel cannot be obtained as the product of the heat kernels of the operators $\mathscr L_1  - \p_t$ and $\mathscr L_2 - \p_t$. Our approach is new and ultimately rests on the generalised Ornstein-Uhlenbeck operators in the opening of H\"ormander's 1967 groundbreaking paper on hypoellipticity.
\end{abstract}

\author{Nicola Garofalo}

\address{Dipartimento d'Ingegneria Civile e Ambientale (DICEA)\\ Universit\`a di Padova\\ Via Marzolo, 9 - 35131 Padova,  Italy}
\vskip 0.2in
\email{nicola.garofalo@unipd.it}

\thanks{Both authors are supported in part by a Progetto SID: ``Non-local Sobolev and isoperimetric inequalities", University of Padova, 2019.}

\author{Giulio Tralli}
\address{Dipartimento d'Ingegneria Civile e Ambientale (DICEA)\\ Universit\`a di Padova\\ Via Marzolo, 9 - 35131 Padova,  Italy}
\vskip 0.2in
\email{giulio.tralli@unipd.it}

\maketitle

\tableofcontents

\section{Introduction}\label{S:intro}

Consider a second order partial differential operator $\mathscr L$ and the heat equation $\mathscr L u - \p_t u =0$ associated with it. Following a well-established tradition by heat kernel we mean a function $p(x,\xi,t)$ such that for any $\xi$ the function $p(\cdot,\xi,\cdot)$ is a solution of the heat equation, and $p(x,\cdot,t)\longrightarrow \delta_x$ in the distributional sense as $t\to 0^+$. 
The aim of this paper is to construct explicit heat kernels for some \emph{hybrid} evolution equations which arise in diverse frameworks such as e.g. sub-Riemannian geometry and problems from the applied sciences that are modelled by some classes of subelliptic equations. By hybrid we mean that the relevant partial differential operator appears in the form $\mathscr L_1 + \mathscr L_2 - \p_t$, but the variables cannot be decoupled. Consequently, the relative heat kernel cannot be simply obtained as the product of the heat kernels of the operators $\mathscr L_1  - \p_t$ and $\mathscr L_2 - \p_t$. Our approach is completely self-contained, elementary, and it is purely based on PDE methods whose final objective is to emphasise the so far unexplored connection of the relevant class of hybrid equations with the generalised operators of Ornstein-Uhlenbeck type in the opening of H\"ormander's groundbreaking 1967 work \cite{Ho}. It is worth mentioning here that as a by-product we obtain a simple proof of the well-known (non-hybrid) cases of the heat operator in a stratified nilpotent Lie group of step two and of the Baouendi-Grushin operator (see respectively Sections \ref{S:CR} and \ref{S:met} below). 

To motivate our results we next discuss some prototypical  examples which fall within the scope of our approach. We begin with an example from conformal CR geometry. In recent years the study of the so-called \emph{extension operators} has received increasing attention from workers in analysis and geometry especially in connection with certain conformally invariant nonlocal operators. A typical situation of interest is represented by the Heisenberg group $\Hn \cong \mathbb C^{n}\times\mathbb R$ with real coordinates $(z,\sigma)$ \footnote{we explicitly mention here that traditionally the letter $t$ is reserved for the vertical variable in $\Hn$. However, since we want to indicate the time variable with $t$, we have opted for the notation $\sigma$. The letter $z$ instead indicates the horizontal variables in $\R^{2n}$.} and horizontal Laplacian 
\begin{equation}\label{slHn}
\mathscr L = \Delta_z + \frac{|z|^2}4 \p_{\sigma\sigma}  + \sum_{i=1}^n \p_\sigma \big(z_i \p_{z_{n+i}} - z_{n+i} \p_{z_i}\big).
\end{equation}
In their seminal paper \cite{FGMT} Frank et al. have introduced the following extension problem:
given a function $u\in C^\infty_0(\Hn)$, find a function $U\in C^\infty(\Hn\times  \R^+_y)$ that solves the Dirichlet problem
\begin{equation}\label{fgmt}
\begin{cases}
\p_{yy} U  + \frac{1-2s}y \p_y U + \frac{y^2}4 \p_{\sigma\sigma} U + \mathscr L U  = 0,\ \ \ \ \ \text{in}\ \Hn\times  \R^+_y,
\\
U(g,0) = u(g),
\end{cases}
\end{equation} 
where the fractional parameter $s\in (0,1)$. The term $\frac{y^2}4 \p_{\sigma\sigma} U$ has a geometric meaning whose explanation comes from the equivalence between \eqref{fgmt} and the \emph{scattering eigenvalue problem} in complex hyperbolic space. 
The second order time-independent PDE in \eqref{fgmt} is a notable example of the type of \emph{hybrid} equations that are the object of interest of the present paper. To clarify this aspect we observe that if we formally think of $w$ as a generic point in the space with fractal dimension  $\R^{2(1-s)}$, and we let $y = |w|$ denote its ``distance" to the origin, then the PDE in \eqref{fgmt} can be interpreted as the action of the differential operator $\Delta_w + \frac{|w|^2}{4} \p_{\sigma\sigma} + \mathscr L$
on functions having spherical symmetry in $w$. If we consider the heat equation associated with such operator,
\begin{equation}\label{hybHn}
\Delta_w U+ \frac{|w|^2}{4} \p_{\sigma\sigma} + \mathscr L U - \p_t U = 0,
\end{equation}
it is immediate to recognise that in such equation the variables $(w,\sigma)\in \R^{2(1-s)}\times \R$ and $g=(z,\sigma)\in\Hn$ cannot be decoupled since the variable $\sigma$ appears in both the limiting operators $\Delta_w + \frac{|w|^2}{4} \p_{\sigma\sigma} - \p_t$ and $\mathscr L-\p_t$ (see \eqref{slHn}). In Section \ref{S:CR} we will show that the heat kernel with pole at the origin associated with \eqref{hybHn} is given by
\begin{align}\label{parBfsk1}
\mathfrak{q}_{(s)}((z,\sigma),t,y) & =  \frac{2}{(4\pi t)^{\frac{m}2 + 2-s}} \int_{\R} e^{- \frac it \sigma\la }   \left(\frac{|\la|}{\sinh |\la|}\right)^{\frac m2+1-s} e^{-\frac{|z|^2 +y^2}{4t}\frac{|\la|}{\tanh |\la|}} d\la.
\end{align}
We emphasise that \eqref{fgmt} is dramatically different from the  extension problem \`a la Caffarelli-Silvestre 
\begin{equation}\label{CS}
\begin{cases}
\p_{yy} U  + \frac{1-2s}y \p_y U  + \mathscr L U  = 0,\ \ \ \ \ \text{in}\ \Hn\times  \R^+_y,
\\
U(g,0) = u(g),
\end{cases}
\end{equation} 
in which the geometric term $\frac{y^2}4 \p_{\sigma\sigma} U $ is missing. The  evolution PDE associated with \eqref{CS} is 
\begin{equation*}
\Delta_w U + \mathscr L U - \p_t U = 0,
\end{equation*}
and it should be clear to the reader that this is not of hybrid type since its fundamental solution 
\[
q^{(s)}((z,\sigma),t,y) = (4\pi t)^{-(1-s)} e^{-\frac{y^2}{4t}} p((z,\sigma),t)
\]
is indeed the product of the fundamental solutions of the two heat operators $\Delta_w  - \p_t$ and $\mathscr L  - \p_t$ (the reader should note that we have used a superscript $(s)$ to distinguish such heat kernel from that in \eqref{parBfsk1}, for which we have used a subscript $(s)$).
Formula \eqref{parBfsk1} (see also the more general case treated in Theorem \ref{T:ext0} below) plays a critical role in the analysis of conformal properties of a certain pseudodifferential operator $\mathscr L_s$ which arises as the Dirichlet-to-Neumann map of \eqref{fgmt}, and we refer the interested reader to the works \cite{RTaim}, \cite{MOZ}, \cite{RT}, \cite{GTfeel}, \cite{GTinter} for more insights into this aspect.

Another significant model of the class of equations encompassed by the present paper is the following:
\begin{equation}\label{one}
\Delta_w f+ \frac{|w|^2}{4} \Delta_\sigma f + \langle w,\nabla_y f\rangle  -  \p_t f = 0,
\end{equation}  
where $w, y \in \R^n$, $\sigma \in \R^k$ and $t>0$. The equation \eqref{one} is a hybrid between the time-dependent Baouendi-Grushin operator in $\Rn\times \R^k\times \R^+_t$, $\Delta_w + \frac{|w|^2}{4} \Delta_\sigma - \p_t$, and the famous degenerate Kolmogorov operator in $\R^{2n}\times \R^+_t$, $\Delta_w + \langle w,\nabla_y \rangle  -  \p_t$. 
There exists no available treatment of \eqref{one} in the literature, but our approach produces the following explicit heat kernel (with pole at a generic point $(w_0,\sigma_0,y_0)\in \R^n\times\R^k\times\R^n$)
\begin{align}\label{two} 
& h((w,\sigma,y),(w_0,\sigma_0,y_0),t)=\\
&=(4\pi)^{-n}\int_{\R^k} e^{2\pi i \left\langle \sigma-\sigma_0,\lambda\right\rangle}e^{-\frac{\pi|\lambda|}{2}(|w|^2-|w_0|^2 + 2nt )} (\det tK_\lambda(t))^{-1/2} \times\notag\\
&\times \exp\bigg\{-\frac{1}{4} \langle (tK_\lambda(t))^{-1}\begin{pmatrix} w_0-e^{-2\pi t|\lambda|}w \\ y_0-y- \frac{1-e^{-2\pi t|\lambda|}}{2\pi|\lambda|}w\end{pmatrix},\begin{pmatrix} w_0-e^{-2\pi t|\lambda|}w \\ y_0-y- \frac{1-e^{-2\pi t|\lambda|}}{2\pi|\lambda|}w\end{pmatrix}\rangle\bigg\} 
d\lambda\notag
\end{align}
where
$$
tK_\lambda(t)=e^{-2\pi t|\lambda|}\begin{pmatrix} \frac{\sinh(2\pi t|\lambda|)}{2\pi|\lambda|}I_n & \frac{\cosh(2\pi t|\lambda|)-1}{4\pi^2|\lambda|^2}I_n\\ \frac{\cosh(2\pi t|\lambda|)-1}{4\pi^2|\lambda|^2}I_n & \frac{e^{2\pi t|\lambda|}(2\pi t |\lambda|-\sinh(2\pi t |\lambda|))+(\cosh(2\pi t |\lambda|)-1)(e^{2\pi t|\lambda|}-1)}{8\pi^3|\lambda|^3}I_n\end{pmatrix}.
$$
Formula \eqref{two}  is a special case of the more general Theorem \ref{T:BK} below, to which we refer the reader.

We now briefly discuss the organisation of the paper. In Section \ref{S:hor} we recall the class \eqref{A0} of generalised Ornstein-Uhlenbeck operators in the opening of H\"ormander's cited paper \cite{Ho}, and for completeness provide a short proof of Proposition \ref{P:hor} since this result constitutes the backbone of the present work. Section \ref{S:met} introduces the hybrid class \eqref{uno}, of which the equation \eqref{one} discussed above is a prototypical  representative. Besides its own interest, such section is instrumental to the rest of the paper. In the subsection \ref{S:ho} we solve the Cauchy problem \eqref{ho1} for a generalised harmonic oscillator. The main result is Proposition \ref{P:Pippo} that establishes a generalisation of the classical formula of Mehler. In Section \ref{S:BGfs} we use this result to derive the heat kernel for the Baouendi-Grushin operator, see 
\eqref{parBfs} in Theorem \ref{T:BGpar}. In Theorem \ref{T:BK} we finally construct the heat kernel for the class of hybrid equations in \eqref{uno}. Section \ref{S:CR} represents the more geometric part of the paper. There we construct the heat kernel for the conformal extension problem \eqref{fgmtG}. The latter represents a time-dependent generalisation to arbitrary groups of Heisenberg type of the above discussed conformal extension problem \eqref{fgmt} from \cite{FGMT}. The main result of the section is Theorem \ref{T:ext0}. To prove it we follow a pattern similar to that in Section \ref{S:met}. We first construct the heat kernel for a generalised harmonic oscillator with a complex drift. This step serves as a building block in the proof of the main Theorem  \ref{T:ext0}. In the process, and as a by-poduct of our approach, we also provide a new elementary proof of the famous formula of Hulanicki-Gaveau-Cygan for the  heat kernel on a Carnot group of step two, see Theorem \ref{T:heat}.


\section{The generalised Ornstein-Uhlenbeck operators of H\"ormander}\label{S:hor}

In this section we recall a well-known explicit heat kernel that constitutes the essential ingredient of the present work. Consider the class of differential equations in $\RN\times (0,\infty)$,
\begin{equation}\label{A0}
\K u = \mathscr A u - \p_t u  \overset{def}{=} \operatorname{tr}(Q \nabla^2 u) +
\langle Bz,\nabla u\rangle  - \p_t u  = 0.
\end{equation}
Here, the $N\times N$ matrices $Q$ and $B$ have real, constant coefficients, and moreover $Q = Q^\star \ge 0$. A basic feature of the operator $\K$ is the invariance with respect to the following non-Abelian group law in $\R^{N+1}$
$$
(z,s)\circ (\zeta,t) = (\zeta+ e^{-tB}z,s+t),
$$
see \cite{LaPo}. We emphasise that the evolution equation $\K u =\sA u  - \p_t u = 0$
encompasses operators that are very different in nature. Besides of course the classical heat equation ($Q = I_N$ and $B = O_N$), it contains
the Ornstein-Uhlenbeck equation $\Delta_z u -
\langle z,\nabla_z u\rangle - \p_t u = 0$ in \cite{OU} ($Q = I_N$ and $B = - I_N$), but also the very degenerate
equation of Kolmogorov from the kinetic theory of
gases $\Delta_v u + \langle v,\nabla_x u\rangle -\p_t u = 0$ in $\R^{2n}\times(0,\infty)$ , see \cite{Kol} ($Q = \begin{pmatrix} I_n & O_n\\ O_n & O_n\end{pmatrix}$ and $B = \begin{pmatrix} O_n & O_n\\I_n & O_n\end{pmatrix}$), as well as the degenerate Ornstein-Uhlenbeck equation $\Delta_v u -\langle v,\nabla_v u\rangle + \langle v,\nabla_x u\rangle -\p_t u = 0$ in $\R^{2n}\times(0,\infty)$ which arises in the Smoluchowski-Kramers approximation of Brownian motion with friction, see \cite{Bri} ($Q = \begin{pmatrix} I_n & O_n\\ O_n & O_n\end{pmatrix}$ and $B = \begin{pmatrix} -I_n & O_n\\I_n & O_n\end{pmatrix}$).

In \cite{Ho} H\"ormander proved that \eqref{A0} is hypoelliptic if and
only if its \emph{covariance matrix} satisfies the following Kalman condition
for one (and therefore every) $t>0$
\begin{equation}\label{Kt}
K(t) = \frac 1t \int_0^t e^{sB} Q e^{s B^\star} ds\ >\ 0.
\end{equation}
The hypothesis \eqref{Kt} will henceforth be assumed in this section. Under such assumption we note that $t\to t K(t)$ is strictly increasing in the sense of quadratic forms: one has in fact for any $t>t_0>0$
\begin{align*}
tK(t)-t_0K(t_0)&= \int_{t_0}^t e^{sB} Q e^{s B^\star} ds = e^{-t_0 B}\left(\int_{0}^{t-t_0} e^{\sigma B} Q e^{\sigma B^\star} d\sigma\right) e^{-t_0 B^\star}\\
&= e^{-t_0 B} (t-t_0)K(t-t_0) e^{-t_0 B^\star} >0.
\end{align*}
It follows that 
\begin{equation}\label{strict}
t \mapsto (tK(t))^{-1}\,\,\,\mbox{ is strictly decreasing.}
\end{equation}
Therefore, the matrix
$$
K^{-1}_\infty=\underset{ t\to \infty}{ \lim }\left( tK(t) \right)^{-1}
$$
is well-defined, and of course it is symmetric and nonnegative definite. Formally, $K^{-1}_\infty$ is the inverse of the matrix $\int_0^\infty e^{sB} Q e^{s B^\star} ds$, but it is well-known that the latter is well-defined if and only if all the eigenvalues of $B$ have strictly negative real part (see, e.g., \cite[Proposition 2.3]{BGTms}). On the other hand $K^{-1}_\infty$ is well-defined for any choice of $Q, B$ satisfying \eqref{Kt}, even if it possibly has a non-trivial kernel.

To introduce the main result of this section we next recall the time-dependent intertwined pseudo-distance
\[
m_t(z,\zeta)  = \sqrt{\langle K(t)^{-1}(\zeta-e^{tB} z ),\zeta-e^{tB}
z \rangle},\ \ \ \ \ \ \ t>0.
\]The behaviour for large $t$ of $tK(t)$ and $m_t(\cdot,\cdot)$ has been instrumental in our previous works \cite{GThls, GTiso} in establishing several functional inequalities related to the differential operator $\mathscr A$. 

\begin{proposition}\label{P:hor}
The heat kernel of \eqref{A0} is given by
\begin{equation}\label{q}
p(z,\zeta,t) = \frac{(4\pi)^{-N/2}}{(\det(t K(t)))^{1/2}} \exp\left( -
\frac{m_t(z,\zeta)^2}{4t}\right).
\end{equation}
More precisely, for any $f\in C(\RN)$ such that
\begin{equation}\label{growthallowed}
e^{-\frac{1}{4}\left\langle K^{-1}_\infty z,z\right\rangle} f(z)\in L^{\infty}\left(\RN\right),
\end{equation}
the function
\begin{equation}\label{ptkolmo}
u(z,t) = P_t f(z)=\int_{\R^N} p(z,\zeta,t) f(\zeta) d\zeta
\end{equation}
solves the Cauchy problem $\K u=  \sA u -\p_t u = 0$ in $\RN\times (0,\infty)$, $u(z,0) = f(z)$.
\end{proposition}

\begin{proof} The proof of \eqref{q} is known and fairly elementary. Denoting by $\xi$ the dual variable of $z$, and letting $\hat u(\xi,t) = \int_{\RN} e^{-2\pi i \langle \xi,z\rangle} u(z,t) dz$, then on the Fourier transform side the Cauchy problem reduces to solving 
\begin{equation}\label{cpFT}
\begin{cases}
\de_t \hat u  + \langle B^\star \xi, \nabla_\xi \hat u\rangle + \left(4 \pi^2
\langle Q\xi,\xi\rangle + \operatorname{tr} B\right) \hat u  = 0\ \ \ \ \ \
\text{in}\ \R^N\times (0,\infty),
\\
\hat u(\xi,0) = \hat{f}(\xi).
\end{cases}
\end{equation}
Now \eqref{cpFT} can be easily solved via the method of characteristics. Fixing $(\xi,t)\in \R^N\times (0,\infty)$, one considers $g(s)  = \hat u(e^{sB^\star} \xi,s+t)$. This function, in turn, solves
\[
g'(s) + \left(4 \pi^2 \langle e^{sB}Q e^{sB^\star}\xi,\xi\rangle +
\operatorname{tr} B\right) g(s) = 0,\ \ \ \ \ \ \ g(0) = \hat u(\xi,t).
\]
Recalling \eqref{Kt}, we see that $g(s)$ is given by
\[
g(s) = \hat u(\xi,t) e^{-4 \pi^2 s \langle K(s)\xi,\xi\rangle} e^{- s \operatorname{tr} B}.
\]
Since $g(-t)=\hat{f}(e^{-tB^\star} \xi)$ and
$K(-t)=e^{-tB}K(t)e^{-tB^*}$, this implies the remarkable formula
\begin{equation}\label{beautyy}
\hat u(\xi,t) = e^{-t \operatorname{tr} B}  e^{- 4t \pi^2
\langle K(t)e^{-tB^*}\xi,e^{-tB^*}\xi\rangle} \hat{f}(e^{-tB^\star} \xi).
\end{equation}
The representation formula \eqref{ptkolmo} follows from \eqref{beautyy} by taking the inverse Fourier transform and
straightforward manipulations. 

It is proved in \cite[Theorem 1.4]{DFP} that, if $|f(z)|\leq C e^{c|z|^2}$ for some positive constants $c, C$, then the function $u=P_t f$ is solution to  $\K u= 0$ in a suitable strip $\R^N\times (0,T)$ and it attains the initial datum $f$. We need to prove that, given $f$ satisfying \eqref{growthallowed}, the function $u=P_t f$ is well defined for every $t>0$ and it defines in fact a solution of the Cauchy problem in the whole $\R^N\times (0,\infty)$. To see this, for any $(z,t)\in \R^N\times (0,\infty)$ we write
\begin{align*}
&u(z,t)=\int_{\R^N}\frac{(4\pi)^{-N/2}}{(\det(t K(t)))^{1/2}} \exp\left( -
\frac{m_t(z,\zeta)^2}{4t}\right) f(\zeta) d\zeta \\
&=\exp^{-\frac{1}{4}\left\langle \left( tK(t) \right)^{-1}e^{tB}z,e^{tB}z\right\rangle}\int_{\R^N}\frac{(4\pi)^{-N/2}}{(\det(t K(t)))^{1/2}} \exp^{\frac{1}{2}\left\langle \left( tK(t) \right)^{-1} \zeta,e^{tB}z\right\rangle}\exp^{-\frac{1}{4}\left\langle \left( tK(t) \right)^{-1} \zeta,\zeta\right\rangle} f(\zeta) d\zeta.
\end{align*}
Exploiting \eqref{growthallowed} we thus infer
\begin{align}\label{boundviainfty}
\left|u(z,t)\right|\leq & \sup_\zeta \left|e^{-\frac{1}{4}\left\langle K^{-1}_\infty \zeta,\zeta\right\rangle} f(\zeta)\right|\exp^{-\frac{1}{4}\left\langle \left( tK(t) \right)^{-1}e^{tB}z,e^{tB}z\right\rangle} \times \\
&\times \int_{\R^N}\frac{(4\pi)^{-N/2}}{(\det(t K(t)))^{1/2}} \exp^{\frac{1}{2}\left\langle \left( tK(t) \right)^{-1} \zeta,e^{tB}z\right\rangle}\exp^{-\frac{1}{4}\left\langle \left( tK(t) \right)^{-1} \zeta,\zeta\right\rangle} e^{\frac{1}{4}\left\langle K^{-1}_\infty \zeta,\zeta\right\rangle}d\zeta.\notag
\end{align}
Property \eqref{strict} ensures, for every fixed $t>0$, the existence of $\mu_t>0$ such that
$$
\left\langle \left(\left( tK(t) \right)^{-1}- K^{-1}_\infty\right)\eta,\eta\right\rangle \geq \mu_t |\eta|^2 \quad\forall\,\eta\in\R^N.
$$
Inserting this information in \eqref{boundviainfty} we deduce that
\begin{align*}
\left|u(z,t)\right|\leq & \sup_\zeta \left|e^{-\frac{1}{4}\left\langle K^{-1}_\infty \zeta,\zeta\right\rangle} f(\zeta)\right|\exp^{-\frac{1}{4}\left\langle \left( tK(t) \right)^{-1}e^{tB}z,e^{tB}z\right\rangle} \times \\
&\times \int_{\R^N}\frac{(4\pi)^{-N/2}}{(\det(t K(t)))^{1/2}} \exp^{\frac{1}{2}\left\langle \left( tK(t) \right)^{-1} \zeta,e^{tB}z\right\rangle}\exp^{-\frac{1}{4}\mu_t|\zeta|^2} d\zeta\\
=& \sup_\zeta \left|e^{-\frac{1}{4}\left\langle K^{-1}_\infty \zeta,\zeta\right\rangle} f(\zeta)\right|\exp^{\frac{1}{4}\left\langle \left(\frac{1}{\mu_t}\left( tK(t) \right)^{-1} - I_N\right) \left( tK(t) \right)^{-1}e^{tB}z,e^{tB}z\right\rangle} \times \\
&\times \int_{\R^N}\frac{(4\pi)^{-N/2}}{(\det(t K(t)))^{1/2}} \exp^{-\frac{1}{4}\left|\sqrt{\mu_t}\zeta - \frac{1}{\sqrt{\mu_t}} \left( tK(t) \right)^{-1} e^{tB}z \right|^2} d\zeta<\infty.
\end{align*}
By arguing in a similar way one can compute the derivatives of $u$ at any point $(z,t)\in \R^N\times (0,\infty)$ by exchanging the order of derivation and integration: this shows that $u$ is solution and completes the proof of the theorem.

\end{proof}

It is easy to see that for the heat equation ($Q = I_N$ and $B = O_N$) the matrix $tK(t)= t I_N$ and therefore $K_\infty^{-1}=O_N$. The matrix $K_\infty^{-1}$ is the null matrix also for the Kolmogorov equation ($Q = \begin{pmatrix} I_n & O_n\\ O_n & O_n\end{pmatrix}$ and $B = \begin{pmatrix} O_n & O_n\\I_n & O_n\end{pmatrix}$) since $tK(t)= \begin{pmatrix} tI_n & \frac{t^2}2 I_n\\ \frac{t^2}2 I_n & \frac{t^3}3 I_n\end{pmatrix}$ and thus $\left(tK(t)\right)^{-1}= \begin{pmatrix} \frac 4t I_n & -\frac{6}{t^2} I_n\\ -\frac{6}{t^2} I_n & \frac{12}{t^3} I_n\end{pmatrix}$. Instead, for the Ornstein-Uhlenbeck equation ($Q = I_N$ and $B = - I_N$) the matrix $tK(t)= \frac{1}{2}(1-e^{-2t})I_N$ and therefore $K_\infty^{-1}=2 I_N$. In the following two examples we discuss two situations that will be useful in the remainder of the present work. Henceforth in this paper we indicate by $j:\R\to \R$ the real analytic function defined by 
\begin{equation}\label{j}
j(\tau)=\frac{\tau}{\sinh(\tau)}.
\end{equation}
Given a $N\times N$ matrix $C$ with real coefficients, the notation $j(C)$ will denote the matrix identified by the power series of the function $j$. It is worth noting that $j(C)$ is invertible with inverse matrix given by $j(C)^{-1}=\sum_{k=0}^{\infty}\frac{C^{2k}}{(2k+1)!}$. Similar interpretation for the matrix $\cosh C$.
\begin{example}[Ornstein-Uhlenbeck with a possibly degenerate drift]\label{ex1} Let $D$ be a real $N\times N$ matrix, that is symmetric and nonnegative definite, and consider the operator obtained by \eqref{A0} with the choice
$$
Q = I_N \qquad\mbox{ and }\qquad B = - 2 D.
$$
Then
\begin{equation}\label{kappameno1ex1}
K_\infty^{-1}  = 4D.
\end{equation}
\begin{proof}
With $Q$ and $B$ given above we have $tK(t)=\int_0^t e^{-4sD} ds$. Keeping in mind that $e^\tau=\sum_{k=0}^{\infty}\frac{\tau^k}{k!}$ and $\sinh(\tau)=\sum_{k=0}^{\infty}\frac{\tau^{2k+1}}{(2k+1)!}$, we now make the following observation
\begin{align*}
e^{tD}\left(tK(t)\right)e^{tD}&=\int_0^t e^{(2t-4s)D} ds=\frac{1}{4}\int_{-2t}^{2t}e^{\tau D} d\tau=\frac{1}{4}\sum_{k=0}^{\infty}\frac{D^{2k}}{(2k)!}\int_{-2t}^{2t}\tau^{2k} d\tau\\
&=t\sum_{k=0}^{\infty}\frac{(2tD)^{2k}}{(2k+1)!}=t\left(j(2tD)\right)^{-1},
\end{align*}
where $j$ is as in \eqref{j}. Since the previous identity can be rewritten as follows
$$
tK(t)=te^{-t D}\left(j(2tD)\right)^{-1}e^{-t D},
$$
we then obtain
\begin{equation}\label{formulainvK}
\left(tK(t)\right)^{-1}=\frac{1}{t}e^{t D}j(2tD)e^{t D}.
\end{equation}
We notice that $\left(tK(t)\right)^{-1}$ and $D$ diagonalize simultaneously, and the function $\frac{e^{2\mu t}}{t}j(2\mu t)$ converges to $4\mu$ as $t\to \infty$ for any $\mu\geq 0$. Then, from \eqref{formulainvK} we obtain 
\[
K_\infty^{-1}=\underset{ t\to \infty}{ \lim }\left( tK(t) \right)^{-1} = 4D,
\]
which proves \eqref{kappameno1ex1}.
\end{proof}
\end{example}

The Smoluchowski-Kramers equation ($Q = \begin{pmatrix} I_n & O_n\\ O_n & O_n\end{pmatrix}$ and $B = \begin{pmatrix} -I_n & O_n\\I_n & O_n\end{pmatrix}$) mentioned in the opening of the section falls within the class considered in the following example.

\begin{example}[Degenerate Ornstein-Uhlenbeck]\label{ex2} Let $N=n+n_1$ with $n,n_1\in N$ and $n_1\leq n$. Consider a symmetric and positive definite $n\times n$ matrix $D_1$ and a $n_1\times n$ matrix $B_0$ with rank $n_1$. For the operator $\K$ in \eqref{A0} corresponding to the choice 
$$
Q= \begin{pmatrix} I_n & O_{n\times n_1}\\ O_{n_1\times n}& O_{n_1\times n_1}\end{pmatrix},\ \ \ \ \ \ \ \hspace{0.5cm}
B= \begin{pmatrix} - 2D_1 & O_{n\times n_1}\\ B_0 & O_{n_1\times n_1}\end{pmatrix},
$$
one has
\begin{equation}\label{kappameno1ex2}
K_\infty^{-1}= \begin{pmatrix} 4D_1 & O_{n\times n_1}\\ O_{n_1\times n} & O_{n_1\times n_1}\end{pmatrix}.
\end{equation}
\begin{proof}
As a first step we observe that with $Q$ and $B$ given above, where the lower indices indicate the dimensions of the various zero matrices, the Kalman condition \eqref{Kt} is satisfied. Since the kernel of $Q$ is $n_1$-dimensional, the operator $\operatorname{tr}(Q \nabla^2 ) + \langle Bz,\nabla \rangle$ is degenerate-elliptic. Denoting $V_i=\p_{z_i}$ for $i\in\{1,\ldots,n\}$ and $V_0 = \langle Bz,\nabla \rangle$ one has
$$
\operatorname{tr}(Q \nabla^2 ) + \langle Bz,\nabla \rangle=\sum_{i=1}^n V_i^2 + V_0.
$$
Moreover, the commutator between $V_i$ and $V_0$ is given by
$$
[V_i,V_0]= \left( B^*\nabla\right)_i=-2\sum_{j=1}^{n} \left(D_1\right)_{ij}\p_{z_j} + \sum_{j=1}^{n_1} \left(B_0\right)_{j i}\p_{z_{n+j}}.
$$
From this relation and the fact that $\mathrm{Im}\left(B_0\right)=\R^{n_1}$ we deduce that the vector fields $V_0, V_1,...V_n$ satisfy H\"ormander's finite rank condition on  the Lie algebra in \cite{Ho} and therefore the operator $\K$ is  hypoelliptic. As we have recalled, this is equivalent to saying that 
\eqref{Kt} hold.

Next, we compute
$$
e^{t B}= \begin{pmatrix} e^{- 2 tD_1} & O_{n\times n_1}\\ \frac{1}{2}B_0 D_1^{-1} \left( I_n-e^{- 2 t D_1}\right) & I_{n_1}\end{pmatrix},
$$
and
$$t K(t) = \begin{pmatrix} K_{11}(t) & K_{12}(t)\\ K^*_{12}(t) & K_{22}(t)\end{pmatrix},$$
where
\begin{align*}
K_{11}(t)&=\frac{1}{4}D_1^{-1} \left( I_n-e^{- 4 t D_1}\right)\\
K_{12}(t)&=\frac{1}{8}D_1^{-2} \left( I_n-e^{- 2 t D_1}\right)^2 B_0^*\\
K_{22}(t)&=\frac{1}{4} B_0 D_1^{-2} \left( tI_n + \frac{1}{4}D_1^{-1} \left(-3 I_n +4 e^{- 2 t D_1}-e^{- 4 t D_1}\right) \right) B_0^*.
\end{align*}
By means of known formulas for the inverse of a partitioned matrix (see, e.g., \cite[Section 0.7.3]{HJ}), we obtain 
\begin{align*}
\left(t K(t)\right)^{-1}=&\begin{pmatrix} \left(K_{11}(t) - K_{12}(t) K^{-1}_{22}(t)K^*_{12}(t)\right)^{-1}  & O_{n\times n_1}\\ O_{n_1\times n} & \left(K_{22}(t) - K^*_{12}(t) K^{-1}_{11}(t)K_{12}(t)\right)^{-1}\end{pmatrix}\times\\
&\times\begin{pmatrix} I_n & -K_{12}(t)K^{-1}_{22}(t)\\ -K^*_{12}(t)K^{-1}_{11}(t) & I_{n_1}\end{pmatrix}.
\end{align*}
We now notice that, as $t\to \infty$, we have the limiting relations 
\[
K^{-1}_{11}(t)\to 4D_1,\ \ \ K_{22}(t)=\frac{t}{4}B_0 D_1^{-2}B_0^* + O(1),\ \ \  K_{12}(t) = O(1).
\]
Since $B_0 D_1^{-2}B_0^*$ is invertible, we  conclude that
\[
K_\infty^{-1}=\underset{ t\to \infty}{ \lim }\left( tK(t) \right)^{-1} = \begin{pmatrix} 4D_1 & O_{n\times n_1}\\ O_{n_1\times n} & O_{n_1\times n_1}\end{pmatrix},
\]
which establishes \eqref{kappameno1ex2}.

\end{proof}
\end{example}


\section{Baouendi met Kolmogorov}\label{S:met}

In this section we discuss a first interesting class of hybrid evolution equations which, remarkably, is directly amenable to the setting of Proposition \ref{P:hor} by means of partial Fourier transform and a suitable exponential transformation, see \eqref{change} below. The reader should bear in mind that the work in this section is also instrumental to the remainder of the paper. Consider an invertible and symmetric $n\times n$ matrix $Q_1$. Let $n_1\leq n$, and fix also a $n_1\times n$ matrix $B_0$ having maximum rank $n_1$. We denote the relevant variables $w\in \Rn, \sigma\in \R^k, y\in \R^{n_1}$ and $t\in (0,\infty)$. Our objective is to solve the Cauchy problem in $\R^n\times\R^k\times\R^{n_1}\times (0,\infty)$,
\begin{equation}\label{uno}
\begin{cases}
\Delta_w f + \frac 14\left\langle Q_1 w,w \right\rangle \Delta_\sigma f  + \langle B_0 w,\nabla_y f\rangle  -  \p_t f = 0,
\\
f(w,\sigma,y,0) = f_0(w,\sigma,y),
\end{cases}
\end{equation}
where $f_0$ is a suitably assigned function in $\R^n\times\R^k\times\R^{n_1}$. We stress that, in our terminology, the partial differential operator $\mathscr L-\p_t$ in \eqref{uno} is hybrid since $\mathscr L = \mathscr L_1 + \mathscr L_2$, where $\mathscr L_1 f=\frac 12 \Delta_w f + \frac 14\left\langle Q_1 w,w \right\rangle \Delta_\sigma f$ and $\mathscr L_2 f= \frac 12 \Delta_w f +\langle B_0 w,\nabla_y f\rangle$, and the term $\Delta_w f$ appears in both $\mathscr L_1$ and $\mathscr L_2$. Before proceeding we emphasise that the hybrid PDE in \eqref{uno} encompasses equations as diverse as the parabolic Baouendi-Grushin equation in $\Rn\times \R^k\times (0,\infty)$ 
\begin{equation}\label{bg}
\Delta_w f + \frac{|w|^2}{4} \Delta_\sigma f - \p_t f = 0,
\end{equation}
see \cite{Ba}, \cite{Gr1, Gr2}, and the already mentioned degenerate Kolmogorov equation in $\R^{2n}\times (0,\infty)$,
\begin{equation}\label{kolmone}
\Delta_w f + \langle w,\nabla_y f\rangle  -  \p_t f = 0,
\end{equation}
see \cite{Kol}.
The former of these two limiting cases is obtained by taking $n_1 = 0$ and $Q_1 = I_n$ in \eqref{uno}, whereas the latter corresponds to taking $k = 0$, $n=n_1$, and $B_0 = I_n$ in \eqref{uno}. To ease the reader's understanding we first discuss in detail our approach to constructing the fundamental solution of \eqref{bg} since this allows to present some of the ideas in a significant, yet simpler model. This will be accomplished in the next two Subsections \ref{S:ho} and \ref{S:BGfs}, the former of which contains a self-contained construction of the Mehler fundamental solution for the generalised harmonic oscillator in \eqref{ho1} below by reducing such operator to a special case of Proposition \ref{P:hor}. We mention that when the matrix $D$ is a multiple of the identity such fundamental solution is well-known and we could have simply lifted its expression from the literature, see Remark \ref{R:mehler} below. In line with the declared self-contained spirit of the present paper, our objective is to show that Proposition \ref{P:Pippo} below can be derived from Proposition \ref{P:hor} by elementary considerations. 


\subsection{The harmonic oscillator aka Ornstein-Uhlenbeck}\label{S:ho}

In what follows given a number $n\in \mathbb N$ we denote by
$D\in M_{n\times n}(\R)$ a matrix such that $D = D^\star$, $D\ge 0$. 
We consider the Cauchy problem for the generalised harmonic oscillator
\begin{equation}\label{ho1}
\begin{cases}
\Delta_z v   - |Dz|^2 v  - \p_t v =  0,
\\
v(z,0) = v_0(z) \ \ \ \ \ \ \ \ \ \ \ \ \ z\in \Rn,\ t>0,
\end{cases}
\end{equation}
where $v_0$ is suitably chosen. We have the following key lemma.

\begin{lemma}\label{hoougen}
Suppose that the functions $v$ and $w$ are connected by the transformation \begin{equation}\label{uw}
v(z,t) = e^{-(\frac 12\langle Dz,z\rangle + t \operatorname{tr} D)} w(z,t).
\end{equation} 
Then, $v$ is a solution to the PDE in \eqref{ho1} if and only if $w$ is a solution to the following equation of Ornstein-Uhlenbeck type
$$
\Delta w - 2 \langle D z,\nabla w\rangle - w_t = 0.
$$
\end{lemma}
 
\begin{proof}
We compute
\begin{align*}
& \nabla(e^{-(\frac 12\langle Dz,z\rangle + t \operatorname{tr} D)}) = - e^{-(\frac 12\langle Dz,z\rangle + t \operatorname{tr} D)}) Dz,
\\
& \Delta(e^{-(\frac 12\langle Dz,z\rangle + t \operatorname{tr} D)}) = (|Dz|^2 - \operatorname{tr} D) e^{-(\frac 12\langle Dz,z\rangle + t \operatorname{tr} D)},
\\
& \p_t (e^{-(\frac 12\langle Dz,z\rangle + t \operatorname{tr} D)}) = - e^{-(\frac 12\langle Dz,z\rangle + t \operatorname{tr} D)} \operatorname{tr} D.
\end{align*}
This gives
\begin{align}\label{expy}
& \Delta v - |Dz|^2 v  - \p_t v
 = e^{-(\frac 12\langle Dz,z\rangle + t \operatorname{tr} D)} (\Delta w - 2 \langle Dz,\nabla w\rangle - w_t).
\end{align}  
The equation \eqref{expy} proves the lemma.
\end{proof}

The next proposition is the main result of this subsection.  

\begin{proposition}[generalised Mehler formula]\label{P:Pippo}
Let $\mathscr M$ be given by the following formula
\begin{align}\label{Pippo}
& \mathscr M(z,\zeta,t)  = (4\pi t)^{-\frac n2} \sqrt{\det j(2tD)} 
\\
& \times \exp\bigg\{-\frac{1}{4t} \bigg(\langle j(2tD) \cosh 2tD\ z, z\rangle + \langle j(2tD) \cosh 2tD\ \zeta,\zeta\rangle - 2 \langle j(2tD)\ z,\zeta \rangle\bigg)\bigg\},
\notag
\end{align}
with $j$ as in \eqref{j}. Then, for any $v_0\in C(\Rn)$ such that
\begin{equation}\label{growthDcase}
e^{-\frac{1}{2}\left\langle D z,z\right\rangle} v_0(z)\in L^{\infty}\left(\Rn\right),
\end{equation}
the function
$$v(z,t) = \int_{\Rn} \mathscr M(z,\zeta,t) v_0(\zeta) d\zeta$$
is solution of \eqref{ho1}.  
\end{proposition}

\begin{proof}
In view of Lemma \ref{hoougen} we see that if $v$ solves the Cauchy problem \eqref{ho1}, then $w = e^{\frac 12\langle D z,z\rangle + \operatorname{tr} D\ t} v$ solves the Cauchy problem
\begin{equation}\label{cpougen}
\begin{cases}
\Delta w - 2 \langle D z,\nabla w\rangle  -  w_t = 0,
\\
w(z,0) = e^{\frac12\langle Dz,z\rangle} v_0(z).
\end{cases}
\end{equation}
To solve \eqref{cpougen} we intend to apply Proposition \ref{P:hor} with $N = n$, $Q = I_n$ and $B = - 2 D$. Keeping Example \ref{ex1} in mind, we know from \eqref{formulainvK} and \eqref{kappameno1ex1} that with this choice we have
\begin{equation}\label{formulainvKbis}
K(t)^{-1}=e^{t D}j(2tD)e^{t D}
\end{equation}
and
$$
K_\infty^{-1}=4D.
$$
Therefore, the initial datum in \eqref{cpougen} is equal to
$$
w(z,0) = e^{\frac12\langle Dz,z\rangle} v_0(z)=e^{\frac14\langle K_\infty^{-1}z,z\rangle} \left(e^{-\frac12\langle Dz,z\rangle} v_0(z)\right)
$$
and, since $v_0$ is continuous and satisfies \eqref{growthDcase}, it nicely fits the assumption of Proposition \ref{P:hor}. According to \eqref{q}-\eqref{ptkolmo} the solution of \eqref{cpougen} is thus given by
$$
w(z,t)=\int_{\Rm} p(z,\zeta,t)e^{\frac12\langle D\zeta,\zeta\rangle} v_0(\zeta) d\zeta,
$$
where we have let  
\begin{align}\label{general}
p(z,\zeta,t)& =(4\pi t)^{-\frac n2}(\det  K(t))^{-1/2} e^{-\frac{1}{4t} \langle (K(t))^{-1}(\zeta-e^{-2t D}z),(\zeta-e^{-2t D}z)\rangle}.
\end{align}
By \eqref{formulainvKbis} we have in particular that
\begin{equation}\label{formuladetK}
(\det K(t))^{-1/2}= e^{t \operatorname{tr} D}\sqrt{\det j(2tD)}.
\end{equation}
Using \eqref{formulainvKbis} and \eqref{formuladetK} in \eqref{general}, we find
\begin{align}\label{general2}
p(z,\zeta,t)& =(4\pi t)^{-\frac n2} e^{t \operatorname{tr} D}\sqrt{\det j(2tD)} e^{-\frac{1}{4t} \langle j(2tD)(e^{t D}\zeta-e^{-t D}z),e^{t D}\zeta-e^{-t D}z\rangle}.
\end{align}
From \eqref{general2} and \eqref{uw} we now see that the solution of \eqref{ho1} is given by
$$
v(z,t) = \int_{\Rn} \mathscr M(z,\zeta,t) v_0(\zeta) d\zeta,
$$
where
\begin{align}\label{werthechmpnsAHAHAH}
\mathscr M(z,\zeta,t) & = (4\pi t)^{-\frac n2} \sqrt{\det j(2tD)} e^{-\frac 1{4t}\langle 2t Dz,z\rangle} e^{\frac 1{4t}\langle 2t D\zeta,\zeta\rangle} 
\\
& \times e^{-\frac{1}{4t} \langle j(2tD)(e^{t D}\zeta-e^{-t D}z),e^{t D}\zeta-e^{-t D}z\rangle}.
\notag
\end{align}
Using the tautological identity
\[
2tD =  j(2tD) \sinh 2tD = j(2tD) \frac{e^{2tD} - e^{-2tD}}2,
\]
and the fact that
\[
e^{\pm t D}  j(2tD)= j(2tD) e^{\pm t D},
\]  
we can now write the argument in the exponentials in \eqref{werthechmpnsAHAHAH} as 
\begin{align*}
& \langle 2 t D z,z\rangle - \langle 2 t D \zeta,\zeta\rangle + \langle j(2tD)(e^{t D}\zeta - e^{-t D} z),e^{t D}\zeta - e^{-t D} z\rangle
\\
& = \langle j(2tD) \sinh 2tD z,z\rangle - \langle j(2tD) \sinh 2tD \zeta,\zeta\rangle + \langle j(2tD)e^{-t D}z,e^{-t D}z\rangle
\\
& + \langle j(2tD)e^{t D}\zeta,e^{t D}\zeta\rangle -  \langle j(2tD)e^{-t D} z,e^{t D}\zeta\rangle -  \langle j(2tD)e^{t D} \zeta,e^{-t D}z\rangle
\\
& = \langle j(2tD) \sinh 2tD z,z\rangle - \langle j(2tD) \sinh 2tD \zeta,\zeta\rangle + \langle j(2tD)e^{-2t D}z,z\rangle
\\
& + \langle j(2tD)e^{2t D}\zeta,\zeta\rangle -  \langle j(2tD) z,\zeta\rangle -  \langle j(2tD)\zeta,z\rangle
\\
& = \langle j(2tD) \cosh 2tD\ z, z\rangle + \langle j(2tD) \cosh 2tD\ \zeta,\zeta\rangle - 2 \langle j(2tD)\ z,\zeta \rangle. 
\end{align*}
This shows \eqref{Pippo}. We have finally proved Proposition \ref{P:Pippo}.

\end{proof}

In what follows we will use the following alternative expression of \eqref{Pippo}
\begin{align}\label{Pippostar}
\mathscr M(z,\zeta,t)  =& (4\pi t)^{-\frac n2} \sqrt{\det j(2tD)} \times\\
&\times \exp\bigg\{-\frac{1}{4t} \bigg( \left| \sqrt{j(2tD) \cosh 2tD}\ \zeta - \sqrt{j(2tD) \cosh^{-1} 2tD}\ z \right|^2 \bigg.\bigg.\notag\\
&+\bigg.\bigg.\langle j(2tD) \left(\cosh 2tD - \cosh^{-1} 2tD\right)\ z, z\rangle \bigg)\bigg\}.\notag
\end{align}
In particular, by performing the change of variable 
\[
\zeta\mapsto \eta= \sqrt{j(2tD) \cosh 2tD}\ \frac{\zeta}{\sqrt{4t}} - \sqrt{j(2tD) \cosh^{-1} 2tD}\ \frac{z}{\sqrt{4t}},
\]
 from \eqref{Pippostar} it is immediate to recognise that
\begin{equation}\label{intpippo}
\int_{\R^n} \mathscr M(z,\zeta,t) d\zeta= \frac{1}{\sqrt{\det \cosh 2tD}}\exp\bigg\{-\frac{1}{4t}\langle j(2tD) \left(\cosh 2tD - \cosh^{-1} 2tD\right)\ z, z\rangle \bigg\}.
\end{equation}
Formula \eqref{intpippo} will be useful in the proof of Theorem \ref{T:BGpar} below.

\begin{remark}\label{R:mehler}
The reader may find it interesting to compare \eqref{Pippo} with the classical 1866 Mehler formula for the harmonic oscillator $
\Delta u - \omega |x|^2 u - u_t = 0$, see e.g. \cite[Section 4.2]{BGV},
$$
\mathscr M(x,y,t) =  (4\pi t)^{-n/2}  \left(\frac{2\sqrt \omega t}{\sinh 2\sqrt \omega t}\right)^{n/2}  e^{-\frac{\sqrt \omega}{2} (\cotanh 2\sqrt \omega t(|x|^2 +|y|^2) - 2 \csch 2\sqrt \omega t \langle x,y\rangle)}.
$$
This formula follows immediately from \eqref{Pippo} by taking $D = \sqrt{\omega} I_n$ in its expression.
\end{remark}


\subsection{The heat kernel of the Baouendi-Grushin operator}\label{S:BGfs}

In his 1967 Ph.D. Dissertation \cite{Ba} under the supervision of B. Malgrange, S. Baouendi first studied the Dirichlet problem in $L^2$ for a class of degenerate elliptic operators that includes the following model
\begin{equation}\label{B}
\Delta_w + \frac{|w|^2}4 \Delta_\sigma,
\end{equation}
where $(w,\sigma)\in \Rn\times \R^k$. At that time M. Vishik was visiting Malgrange, who discussed with him the thesis project of Baouendi. Vishik subsequently asked Malgrange permission to suggest to his own Ph.D. student, Grushin, to work on some questions related to the hypoellipticity of operators modelled on \eqref{B}, see \cite{Gr1, Gr2}. This is how the operator \eqref{B} became known as the \emph{Baouendi-Grushin operator}. This operator is also important since it is connected to harmonic functions with special symmetries in a group of Heisenberg type $\bG$. 
We notice, in this respect, that there is no global group law underlying \eqref{B}, but the operator is invariant with respect to standard translations $(w,\sigma)\to (w,\sigma +\sigma')$ along the manifold of degeneracy $M = \{0\}\times\R^k$. We consider the Cauchy problem 
\begin{equation}\label{parb}
\begin{cases}
\Delta_w u + \frac{|w|^2}4 \Delta_\sigma u - \p_t u = 0\ \ \ \ \ \ \ \text{in}\ \R^{n+k}\times (0,\infty),
\\
u((w,\sigma),0) = f(w,\sigma).
\end{cases}
\end{equation}
The next result provides an explicit heat kernel for \eqref{B}. 

\begin{theorem}\label{T:BGpar}
Let 
\begin{align}\label{parBfs}
 \mathscr B((w,\sigma),(w',\sigma'),t) & =  \frac{2^k}{(4\pi t)^{\frac{n}2 +k}} \int_{\R^k} e^{- \frac it \langle \la,\sigma'-\sigma\rangle}   \left(\frac{|\la|}{\sinh |\la|}\right)^{\frac n2} 
\\
& \times  e^{-\frac{|\la|}{4t \tanh |\la|} ((|w|^2 +|w'|^2) - 2 \langle w,w'\rangle \sech |\la|)} d\la.
\notag\end{align}
Then for every $f\in \mathscr S(\R^{n+k})$ the function
$$u((w,\sigma),t) = \int_{\Rn} \int_{\R^k} \mathscr B((w,\sigma),(w',\sigma'),t) f(w',\sigma') dw' d\sigma'$$
is a solution of \eqref{parb}.
\end{theorem}

\begin{proof}
We indicate with $\hat u(w,\la,t) = \int_{\R^k} e^{-2\pi i\langle \la,\sigma\rangle} u(w,\sigma,t) d\sigma$ the partial Fourier transform of $u$ with respect to the variable $\sigma\in \R^k$, with dual variable $\la\in \R^k$. Applying such Fourier transform to \eqref{parb}, for any fixed $\la\in \R^k$ we obtain 
$$\begin{cases}
\Delta_w \hat u -  \pi^2 |\la|^2 |w|^2 \hat u - \p_t \hat u = 0\ \ \ \ \ \ \ \text{in}\ \R^{n}\times (0,\infty),
\\
\hat u((w,\la),0) = \hat f(w,\la).
\end{cases}$$
This is a Cauchy problem for a harmonic oscillator such as \eqref{ho1} above, with matrix $D = D(\la) = \pi |\la| I_n$. From Proposition \ref{P:Pippo} we know that the solution of such problem is given by the formula
\[
\hat u((w,\la),t) = \int_{\Rn} \mathscr M_\la(w,w',t) \hat f(w',\la) dw' = \int_{\R^k} e^{-2\pi i\langle \la,\sigma'\rangle}  \int_{\Rn} \mathscr M_\la(w,w',t) f(w',\sigma') dw' d\sigma',
\]
where $\mathscr M_\la(w,w',t)$ is Mehler's fundamental solution given by 
\begin{equation}\label{solho3}
\mathscr M_\la(w,w',t)  =  (4\pi)^{-n/2}  \left(\frac{2\pi |\la|}{\sinh 2\pi t|\la|}\right)^{n/2} 
   e^{-\frac{\pi |\la|}{2} ((|w|^2 +|w'|^2) \cotanh(2 \pi t |\la|) - 2 \langle w,w'\rangle \csch(2\pi t |\la|))},
\end{equation}
see Remark \ref{R:mehler}. We know that $\hat u((w,\la),t) \underset{t\to 0^+}{\longrightarrow} \hat f(w,\la)$ in the pointwise sense. We will now show that, for every fixed $w\in \Rn$, such convergence also holds in $L^1(\R^k, d \lambda)$. To see this we write 
\begin{align*}
&\int_{\R^k}|\hat u((w,\la),t)-\hat f(w,\la)|d\lambda \\
&= \int_{\R^k}\left|\hat u((w,\la),t)-\hat f(w,\la) \left(\int_{\Rn} \mathscr M_\la(w,w',t) dw' + 1 -\int_{\Rn} \mathscr M_\la(w,w',t) dw'\right)\right|d\lambda\\
&\leq  \int_{\R^k}\int_{\Rn} \mathscr M_\la(w,w',t) \left|\hat f(w',\la) - \hat f(w,\la)\right|  dw'd\lambda + \int_{\R^k} \left| 1 -\int_{\Rn} \mathscr M_\la(w,w',t) dw'\right| \left|\hat f(w,\la)\right|d\lambda.
\end{align*}
Applying \eqref{intpippo} with $D=\pi|\lambda|I_n$ we easily obtain
$$
\int_{\Rn} \mathscr M_\la(w,w',t) dw'  = \left(\frac{1}{\cosh 2\pi t|\lambda|}\right)^{\frac{n}{2}}e^{-\frac{|w|^2}{4t}\frac{2\pi t|\lambda|}{\tanh 2\pi t |\lambda|}\left(1-\sech^2 2\pi t|\lambda|\right) }.
$$
From this identity we immediately see that
\begin{itemize}
\item[(i)] $0\leq \int_{\Rn} \mathscr M_\la(w,w',t) dw'  \le 1$;
\item[(ii)] $\int_{\Rn} \mathscr M_\la(w,w',t) dw' \to 1$ as $t\to 0^+$.
\end{itemize}
From (i) and (ii) we infer by dominated convergence theorem that
$$
\int_{\R^k} \left| 1 -\int_{\Rn} \mathscr M_\la(w,w',t) dw'\right| \left|\hat f(w,\la)\right|d\lambda \underset{t\to 0^+}{\longrightarrow} 0.
$$
On the other hand, by applying \eqref{Pippostar} with $D=\pi|\lambda|I_n$ and performing the change of variables $w'\mapsto \eta$ with $w'-w\sech 2\pi t|\lambda|=\eta\sqrt{4t\frac{\tanh 2\pi t |\lambda|}{2\pi t|\lambda|}}$, we deduce
\begin{align*}
&\int_{\R^k}\int_{\Rn} \mathscr M_\la(w,w',t) \left|\hat f(w',\la) - \hat f(w,\la)\right|  dw'd\lambda\\
&=\pi^{-\frac{n}{2}}\int_{\R^k}\int_{\Rn} \left(\frac{1}{\cosh 2\pi t|\lambda|}\right)^{\frac{n}{2}} e^{-\frac{|w|^2}{4t}\frac{2\pi t|\lambda|}{\tanh 2\pi t |\lambda|}\left(1-\sech^2 2\pi t|\lambda|\right) } e^{-|\eta|^2} \times\\
&\times \left|\hat f\left(w\sech 2\pi t|\lambda|+\eta\sqrt{4t\frac{\tanh 2\pi t |\lambda|}{2\pi t|\lambda|}},\la\right) - \hat f(w,\la)\right|  d\eta d\lambda \\
&\leq \pi^{-\frac{n}{2}}\int_{\R^k}\int_{\Rn} e^{-|\eta|^2} \left|\hat f\left(w\sech 2\pi t|\lambda|+\eta\sqrt{4t\frac{\tanh 2\pi t |\lambda|}{2\pi t|\lambda|}},\la\right) - \hat f(w,\la)\right|  d\eta d\lambda \underset{t\to 0^+}{\longrightarrow} 0,
\end{align*}
where the last limiting relation can be again justified via dominated convergence theorem since $\hat f \in\mathscr S$ and therefore $\hat f(w',\lambda)$ is continuous at $(w,\lambda)$ and it can be bounded by an integrable function uniformly in $w'$. This proves that
\begin{equation}\label{lunogrus}
\underset{t\to 0^+}{\lim} \int_{\R^k}|\hat u((w,\la),t)-\hat f(w,\la)|d\lambda =0 \qquad\forall\, w\in\R^n.
\end{equation}
If we now take the inverse Fourier transform of $\hat u((w,\la),t)$, we find the following representation for the solution of problem \eqref{parb}
\begin{equation}\label{solho4}
u((w,\sigma),t) = \int_{\Rn} \int_{\R^k} \left(\int_{\R^k} e^{-2\pi i\langle \la,\sigma'-\sigma\rangle}  \mathscr M_\la(w,w',t) d\la\right) f(w',\sigma') dw' d\sigma'.
\end{equation}
We stress that \eqref{lunogrus} ensures uniform, and therefore pointwise convergence of $u((w,\sigma),t)$ to $f(w,\sigma)$ as $t\to 0^+$. From \eqref{solho4} it is clear that the heat kernel of the parabolic Baouendi-Grushin equation in \eqref{parb} is thus given by
\begin{equation}\label{parb4}
\mathscr B((w,\sigma),(w',\sigma'),t) = \int_{\R^k} e^{-2\pi i\langle \la,\sigma'-\sigma\rangle}  \mathscr M_\la(w,w',t) d\la,
\end{equation}
where $\mathscr M_\la(w,w',t)$ is as in \eqref{solho3}. Changing variable $\la \to 2\pi t \la$ in the integral over $\R^k$, we finally obtain \eqref{parBfs} from \eqref{parb4}.

\end{proof}

We mention the works \cite{CCGGL, BFIh, CCFI} for various derivations and integral representations of the heat kernel for \eqref{B} when $k=1$. For related discussions about the fundamental solutions of more general (time-independent) Baouendi-Grushin operators we refer the reader to \cite{Gjde, BGG2, BFI}.


\subsection{Back to the hybrid equation }\label{S:back}

In this subsection we finally solve the Cauchy problem \eqref{uno}. In preparation for our main result, Theorem \ref{T:BK} below, we introduce for every $\la\in \R^k$ the $(n+n_1)\times (n+n_1)$ matrices
\begin{equation}\label{QB}
Q= \begin{pmatrix} I_n & O_{n\times n_1}\\ O_{n_1\times n}& O_{n_1\times n_1}\end{pmatrix},\ \ \ \ \ \ \ \hspace{0.5cm}
B_\la = \begin{pmatrix} - 2\pi|\lambda|\sqrt{Q_1} & O_{n\times n_1}\\ B_0 & O_{n_1\times n_1}\end{pmatrix}.
\end{equation}
For every $t>0$ we next consider the covariance matrix associated with $Q$ and $B_\la$
\begin{equation}\label{Kla}
tK_\lambda(t)\overset{def}{=}\int_0^te^{sB_\la}Qe^{sB_\la^*} ds.
\end{equation}
If we keep in mind Example \ref{ex2} with the choice $D_1=\pi|\lambda|\sqrt{Q_1}$, we know that the matrix $K_\lambda(t)$ is positive definite for every $t>0$, i.e. it satisfies the Kalman condition \eqref{Kt} above. Moreover, by \eqref{kappameno1ex2} we have
\begin{equation}\label{eccolamaldida}
K_{\la,\infty}^{-1}=\begin{pmatrix} 4 \pi|\lambda|\sqrt{Q_1} & O_{n\times n_1}\\ O_{n_1\times n} & O_{n_1\times n_1}\end{pmatrix}.
\end{equation}
We next establish a lemma that will play a crucial role in the proof of Theorem \ref{T:BK}.
\begin{lemma}\label{wmatrix}
For any $t>0$ and $\lambda\in\R^k$ we have
\begin{align*}
(tK_\lambda(t))^{-1}+\frac{1}{2}e^{-tB^*_\la}K_{\la,\infty}^{-1}e^{-tB_\la} \geq \left(tK_\lambda(t) - \frac{1}{2} tK_\lambda(t) K_{\la,\infty}^{-1} tK_\lambda(t)\right)^{-1}.
\end{align*}
\end{lemma}
\begin{proof}
The case $\lambda=0$ is trivial since $tK_\lambda(t)$ is still a positive definite matrix and $K_{\la,\infty}^{-1}$ is the null matrix. We can thus assume $\lambda\neq 0$. Keeping in mind the explicit form of $e^{-tB_\la}$ (see Example \ref{ex2} with $D_1=\pi |\lambda| \sqrt{Q_1}$), we notice that
$$
\frac{1}{2}e^{-tB^*_\la}K_{\la,\infty}^{-1}e^{-tB_\la}=\begin{pmatrix} 2 \pi|\lambda|\sqrt{Q_1}e^{4 \pi t|\lambda|\sqrt{Q_1}} & O_{n\times n_1}\\ O_{n_1\times n} & O_{n_1\times n_1}\end{pmatrix}
$$
is a $(n+n_1)\times (n+n_1)$ matrix of rank $n$. We can then exploit the formula for the inverse of a small-rank adjustment (see, e.g., \cite[Section 0.7.4]{HJ}), which is sometimes referred to as the Sherman-Morrison-Woodbury formula, to infer that
\begin{align*}
&\left( (tK_\lambda(t))^{-1}+\frac{1}{2}e^{-tB^*_\la}K_{\la,\infty}^{-1}e^{-tB_\la}\right)^{-1} \\
&= tK_\lambda(t) - tK_\lambda(t) \begin{pmatrix} \left( \frac{1}{2 \pi|\lambda|}Q_1^{-\frac{1}{2}}e^{-4 \pi t|\lambda|\sqrt{Q_1}} + \frac{1}{4 \pi|\lambda|}Q_1^{-\frac{1}{2}}\left(I_n - e^{-4 \pi t|\lambda|\sqrt{Q_1}}\right) \right)^{-1} & O_{n\times n_1}\\ O_{n_1\times n} & O_{n_1\times n_1}\end{pmatrix} tK_\lambda(t),
\end{align*}
where we have used the fact that the first block $K_{11}(t)$ in $tK_\lambda(t)$ is equal to 
$$\frac{1}{4 \pi|\lambda|}Q_1^{-\frac{1}{2}}\left(I_n - e^{-4 \pi t|\lambda|\sqrt{Q_1}}\right)$$
(see again Example \ref{ex2}). By exploiting the simple inequality $\left(I_n + e^{-4 \pi t|\lambda|\sqrt{Q_1}}\right)^{-1}\geq \frac{1}{2}I_n$, we deduce that
\begin{align*}
&\left( (tK_\lambda(t))^{-1}+\frac{1}{2}e^{-tB^*_\la}K_{\la,\infty}^{-1}e^{-tB_\la}\right)^{-1} \\
&=tK_\lambda(t) - tK_\lambda(t) \begin{pmatrix} 4\pi |\lambda|\sqrt{Q_1}\left(I_n + e^{-4 \pi t|\lambda|\sqrt{Q_1}}\right)^{-1} & O_{n\times n_1}\\ O_{n_1\times n} & O_{n_1\times n_1}\end{pmatrix} tK_\lambda(t)\\
&\leq tK_\lambda(t) - tK_\lambda(t) \begin{pmatrix} 2\pi |\lambda|\sqrt{Q_1} & O_{n\times n_1}\\ O_{n_1\times n} & O_{n_1\times n_1}\end{pmatrix} tK_\lambda(t) =  tK_\lambda(t) - \frac{1}{2} tK_\lambda(t) K_{\la,\infty}^{-1} tK_\lambda(t)
\end{align*}
which implies the desired conclusion.

\end{proof}

{\allowdisplaybreaks
\begin{theorem}\label{T:BK}
For $X=(w,y)$ and $X_0=(w_0,y_0)$ we let
$$
p_\la((w,y),(w_0,y_0),t) =(4\pi)^{-\frac{n+n_1}{2}}(\det tK_\lambda(t))^{-1/2} e^{-\frac{1}{4} \langle (tK_\lambda(t))^{-1}\left(X_0-e^{tB_\la}X\right), X_0-e^{tB_\la}X\rangle},
$$
and define
\begin{align}\label{kernelBK}
h((w,\sigma,y),(w_0,\sigma_0,y_0),t) = \int_{\R^k} & e^{2\pi i \left\langle \sigma-\sigma_0,\lambda\right\rangle}e^{-\pi|\lambda|(\frac 12 \left\langle \sqrt{Q_1}(w+w_0),w-w_0 \right\rangle + \operatorname{tr} \sqrt{Q_1} t )} 
\\
& \times p_\la((w,y),(w_0,y_0),t)\ d\la.
\notag
\end{align}
Given $f_0\in \mathscr S(\R^n\times\R^k\times\R^{n_1})$, the function
\begin{align}\label{unocp}
f(w,\sigma,y,t) & = \int_{\R^n\times\R^k\times\R^{n_1}}  h((w,\sigma,y),(w_0,\sigma_0,y_0),t)\ f_0(w_0,\sigma_0,y_0)\     dw_0 d\sigma_0 dy_0,
\end{align}
solves the Cauchy problem \eqref{uno}. 
\end{theorem}

\begin{proof}
As before, our first step is to take the Fourier transform of \eqref{uno} with respect to the variable $\sigma$, with dual variable $\la\in \R^k$. If we let  
\[
v_\lambda(w,y,t)=\hat{f}(w,\lambda,y,t) = \int_{\R^k} e^{-2\pi i \langle \la,\sigma\rangle} f(w,\sigma,y,t) d\sigma,
\]
then for every fixed $\la\in \R^k$ the problem \eqref{uno} becomes in $\R^n\times\R^{n_1}\times (0,\infty)$
\begin{equation}\label{due}
\begin{cases}
\Delta_w v_\lambda  - \pi^2|\lambda|^2\left\langle Q_1 w,w \right\rangle  v_\lambda + \langle B_0 w,\nabla_y v_\lambda\rangle  -  \partial_t v_\lambda = 0,
\\
v_\lambda(w,y,0) = \hat{f}_0(w,\lambda,y), \ \ \ \ \ \ \ \ \ \ \ \ \ (w,y)\in \R^n\times\R^{n_1}.
\end{cases}
\end{equation}
Our second step is to make the following change of dependent variable $v_\la \to u_\la$, where the two functions are linked by the relation
\begin{equation}\label{change}
v_\lambda(w,y,t) = e^{-\pi|\lambda|(\frac 12 \left\langle \sqrt{Q_1} w,w \right\rangle + \operatorname{tr}\sqrt{Q_1} t )} u_\la(w,y,t).
\end{equation}
This step represents a generalised version of \eqref{uw} in Lemma \ref{hoougen}. 
After some straightforward computations one recognises that in terms of the function $u_\la$ the problem \eqref{due} becomes in $\R^n\times\R^{n_1}\times (0,\infty)$
\begin{equation}\label{tre}
\begin{cases}
\Delta_w u_\la - 2\pi|\lambda| \langle  \sqrt{Q_1} w,\nabla_w u_\la\rangle + \langle B_0 w,\nabla_y u_\la\rangle  -  \partial_t u_\la = 0,
\\
u_\la(w,y,0) = e^{\frac{\pi}2 |\lambda| \left\langle \sqrt{Q_1} w,w \right\rangle}\hat{f}_0(w,\lambda,y), \ \ \ \ \ \ \ \ \ \ \ \ \ (w,y)\in \R^n\times\R^{n_1}.
\end{cases}
\end{equation}
Remarkably, the PDE in \eqref{tre} can be cast in the form \eqref{A0}, where now $N = n + n_1$, and the matrices $Q$ and $B = B_\la$ are given by \eqref{QB}. The covariance matrix is given by the positive definite matrix $K_\la(t)$ in \eqref{Kla}. By \eqref{eccolamaldida} we can rewrite the initial datum in \eqref{tre} as
$$u_\la(w,y,0) = e^{\frac{\pi}2 |\lambda| \left\langle \sqrt{Q_1} w,w \right\rangle}\hat{f}_0(w,\lambda,y)=e^{\frac{1}{8} \left\langle K_{\la,\infty}^{-1} (w,y),(w,y)\right\rangle}\hat{f}_0(w,\lambda,y).$$
Since $\hat{f}_0$ is bounded, we can apply Proposition \ref{P:hor}: if we thus let $X=(w,y)$ and $X_0=(w_0,y_0)$, and define $p_\la((w,y),(w_0,y_0),t)$ as in \eqref{q} above, we infer that the function
$$
u_\la(w,y,t)=\int_{\R^n\times\R^{n_1}}p_\la((w,y),(w_0,y_0),t)e^{\frac 12 \pi|\lambda| \left\langle \sqrt{Q_1} w_0,w_0 \right\rangle}\hat{f}_0(w_0,\lambda,y_0) dw_0dy_0
$$
solves the problem \eqref{tre}. In view of \eqref{change} this implies that 
\begin{align*}
v_\lambda(w,y,t)=&e^{-\pi|\lambda|(\frac 12 \left\langle \sqrt{Q_1} w,w \right\rangle + \operatorname{tr} \sqrt{Q_1} t )}\times\\
&\times \int_{\R^n\times\R^{n_1}}p_\la((w,y),(w_0,y_0),t)e^{\frac 12 \pi|\lambda| \left\langle \sqrt{Q_1} w_0,w_0 \right\rangle}\hat{f}_0(w_0,\lambda,y_0) dw_0dy_0.
\end{align*}
As we intend to take the inverse Fourier transform of $v_\lambda(w,y,t)$ with respect to $\lambda$, we want to understand the behaviour of $v_\lambda(w,y,t)$ with respect to this variable. Since $f_0$ belongs to the Schwartz class, we know that $\hat{f}_0(w_0,\lambda,y_0)$ decays faster than any polynomial in $\lambda$ in a uniform way with respect to $(w_0,y_0)$. Our objective is thus to analyse the function
$$
I_{X}(\lambda,t)=e^{-\pi|\lambda|(\frac 12 \left\langle \sqrt{Q_1} w,w \right\rangle + \operatorname{tr} \sqrt{Q_1} t )} \int_{\R^n\times\R^{n_1}}p_\la((w,y),(w_0,y_0),t)e^{\frac 12 \pi|\lambda| \left\langle \sqrt{Q_1} w_0,w_0 \right\rangle} dw_0dy_0.
$$
Using \eqref{eccolamaldida} and the explicit expression of $p_\la(X,X_0,t)$ in \eqref{q}, we write
\begin{align}\label{iiii}
&I_{X}(\lambda,t)=\frac{e^{-\pi t|\lambda|\operatorname{tr} \sqrt{Q_1}}}{ (4\pi)^{\frac{n+n_1}{2}} }\frac{e^{-\frac 18 \left\langle K_{\la,\infty}^{-1} X,X \right\rangle}}{\sqrt{\det tK_\lambda(t)}} \times \\
&\times \int_{\R^{n+n_1}}e^{-\frac{1}{4} \langle (tK_\lambda(t))^{-1}\left(X_0-e^{tB_\la}X\right), X_0-e^{tB_\la}X\rangle}e^{\frac 18 \left\langle K_{\la,\infty}^{-1} X_0,X_0 \right\rangle} dX_0 \notag\\
&=\frac{e^{-\pi t|\lambda|\operatorname{tr} \sqrt{Q_1}}}{ (4\pi)^{\frac{n+n_1}{2}} }\frac{e^{-\frac 14 \left\langle \left( e^{tB^*_\la}(tK_\lambda(t))^{-1}e^{tB_\la}+\frac{1}{2}K_{\la,\infty}^{-1}\right) X,X \right\rangle}}{\sqrt{\det tK_\lambda(t)}}\times\notag\\
&\times e^{\frac 14 \left\langle \left(I -\frac{1}{2}\left(tK_\lambda(t)\right)^{\frac{1}{2}}K_{\la,\infty}^{-1}\left(tK_\lambda(t)\right)^{\frac{1}{2}}\right)^{-1} \left(tK_\lambda(t)\right)^{-\frac{1}{2}} e^{tB_\la} X,\left(tK_\lambda(t)\right)^{-\frac{1}{2}} e^{tB_\la} X \right\rangle}\times\notag\\
&\times \int_{\R^{n+n_1}}\exp\left\{-\frac{1}{4}\left|\left(I -\frac{1}{2}\left(tK_\lambda(t)\right)^{\frac{1}{2}}K_{\la,\infty}^{-1}\left(tK_\lambda(t)\right)^{\frac{1}{2}}\right)^{\frac{1}{2}}\left(tK_\lambda(t)\right)^{-\frac{1}{2}}X_0 + \right.\right. \notag\\
&\left.\left.\hspace{1.5cm}-\left(I -\frac{1}{2}\left(tK_\lambda(t)\right)^{\frac{1}{2}}K_{\la,\infty}^{-1}\left(tK_\lambda(t)\right)^{\frac{1}{2}}\right)^{-\frac{1}{2}} \left(tK_\lambda(t)\right)^{-\frac{1}{2}} e^{tB_\la} X\right|^2\right\} dX_0\notag\\
&=e^{-\pi t|\lambda|\operatorname{tr} \sqrt{Q_1}}\left(\det \left(I -\frac{1}{2}\left(tK_\lambda(t)\right)^{\frac{1}{2}}K_{\la,\infty}^{-1}\left(tK_\lambda(t)\right)^{\frac{1}{2}}\right)\right)^{-\frac{1}{2}}  \times\notag\\
&\times e^{-\frac 14 \left\langle \left( (tK_\lambda(t))^{-1}+\frac{1}{2}e^{-tB^*_\la}K_{\la,\infty}^{-1}e^{-tB_\la}\right) e^{tB_\la}X,e^{tB_\la}X \right\rangle}\times\notag\\
&\times e^{\frac 14 \left\langle \left(\left(tK_\lambda(t)\right) -\frac{1}{2}\left(tK_\lambda(t)\right)K_{\la,\infty}^{-1}\left(tK_\lambda(t)\right)\right)^{-1} e^{tB_\la} X,e^{tB_\la} X \right\rangle}\notag
\end{align}
where in the last equality we have used the change of variables $X_0\mapsto Z$, where 
\begin{align*}
Z&=\left(I -\frac{1}{2}\left(tK_\lambda(t)\right)^{\frac{1}{2}}K_{\la,\infty}^{-1}\left(tK_\lambda(t)\right)^{\frac{1}{2}}\right)^{\frac{1}{2}}\left(4tK_\lambda(t)\right)^{-\frac{1}{2}}X_0 +\\
&-\left(I -\frac{1}{2}\left(tK_\lambda(t)\right)^{\frac{1}{2}}K_{\la,\infty}^{-1}\left(tK_\lambda(t)\right)^{\frac{1}{2}}\right)^{-\frac{1}{2}} \left(4tK_\lambda(t)\right)^{-\frac{1}{2}} e^{tB_\la} X,
\end{align*}
and the fact that $\int_{\R^{n+n_1}}e^{-|Z|^2}dZ = \pi^{\frac{n+n_1}{2}}$. At this point a small miracle happens since from Lemma \ref{wmatrix} we have, for every $X\in\R^{n+n_1}, \lambda\in\R^k$, and $t>0$, that
\begin{align*}
&\left\langle \left( (tK_\lambda(t))^{-1}+\frac{1}{2}e^{-tB^*_\la}K_{\la,\infty}^{-1}e^{-tB_\la}\right) e^{tB_\la}X,e^{tB_\la}X \right\rangle\\
&\geq \left\langle \left(\left(tK_\lambda(t)\right) -\frac{1}{2}\left(tK_\lambda(t)\right)K_{\la,\infty}^{-1}\left(tK_\lambda(t)\right)\right)^{-1} e^{tB_\la} X,e^{tB_\la} X \right\rangle.
\end{align*}
Inserting this information in \eqref{iiii} we obtain
\begin{equation}\label{chebound}
0\leq I_{X}(\lambda,t)\leq \frac{e^{-\pi t|\lambda|\operatorname{tr} \sqrt{Q_1}}}{ \left(\det \left(I -\frac{1}{2}\left(tK_\lambda(t)\right)^{\frac{1}{2}}K_{\la,\infty}^{-1}\left(tK_\lambda(t)\right)^{\frac{1}{2}}\right)\right)^{\frac{1}{2}} } \leq 2^{\frac{n+n_1}{2}} e^{-\pi t|\lambda|\operatorname{tr} \sqrt{Q_1}},
\end{equation}
where in the last inequality we have exploited \eqref{strict} to obtain $\left(I -\frac{1}{2}\left(tK_\lambda(t)\right)^{\frac{1}{2}}K_{\la,\infty}^{-1}\left(tK_\lambda(t)\right)^{\frac{1}{2}}\right)\geq \frac{1}{2}I$. Moreover we also have
\begin{equation}\label{IXt0}
\underset{t\to 0^+}{\lim} I_X(\lambda, t) =1\qquad\,\, \forall\, X\in\R^{n+n_1},\, \lambda\in\R^k.
\end{equation}
The limiting behaviour in \eqref{IXt0} can be checked by using the change of variables $X_0\mapsto Z$ where $X_0=e^{tB_\lambda}X + (4tK_\lambda(t))^{\frac{1}{2}}Z$ in the definition of $I_X(\lambda, t)$, as this gives
\begin{align*}
I_{X}(\lambda,t)=&\frac{e^{-\pi t|\lambda|\operatorname{tr} \sqrt{Q_1}}}{ \pi^{\frac{n+n_1}{2}} }e^{-\frac 18 \left\langle K_{\la,\infty}^{-1} X,X \right\rangle}\times\\
&\times\int_{\R^{n+n_1}}e^{-|Z|^2+\frac{1}{8} \left\langle K_{\la,\infty}^{-1}\left(e^{tB_\lambda}X + (4tK_\lambda(t))^{\frac{1}{2}}Z \right),\left(e^{tB_\lambda}X + (4tK_\lambda(t))^{\frac{1}{2}}Z \right) \right\rangle} dZ.
\end{align*}
Since $e^{tB_\lambda}X + (4tK_\lambda(t))^{\frac{1}{2}}Z\to X$ as $t\to 0^+$, we easily obtain \eqref{IXt0} from the above identity. 
Hence, by exploiting \eqref{chebound} and \eqref{IXt0}, we can argue as in the proof of \eqref{lunogrus} in Theorem \ref{T:BGpar} to deduce that 
\begin{equation}\label{sitornaa}
\int_{\R^{k}}\left|v_\lambda(w,y,t)-\hat{f}_0(w,\lambda,y)\right|d\lambda \underset{t\to 0^+}{\longrightarrow} 0.
\end{equation}
Keeping in mind that $v_\lambda(w,y,t)=\hat{f}(w,\lambda,y,t)$, we are now ready to take the inverse Fourier transform, obtaining  
\begin{align*}
f(w,\sigma,y,t)=\int_{\R^k}\int_{\R^n\times\R^{n_1}}\int_{\R^k}&  e^{2\pi i \left\langle \sigma,\lambda\right\rangle}e^{-\pi|\lambda|(\frac 12 \left\langle \sqrt{Q_1} w,w \right\rangle + \operatorname{tr} \sqrt{Q_1} t )}p_\la((w,y),(w_0,y_0),t)\times\\
&\times e^{\frac 12 \pi|\lambda| \left\langle \sqrt{Q_1} w_0,w_0 \right\rangle} e^{-2\pi i \left\langle \sigma_0,\lambda\right\rangle}f_0(w_0,\sigma_0,y_0) d\sigma_0 dw_0dy_0 d\lambda.
\end{align*}
Since $f_0\in \mathscr{S}$ and thanks to \eqref{chebound}, $f(w,\sigma,y,t)$ is a well-defined and smooth function and it coincides with the expression stated in \eqref{unocp}. Proceeding verbatim  as in the proof of Theorem \ref{T:BGpar} and using \eqref{sitornaa}, we also see that $f$ is solution of the Cauchy problem \eqref{uno}. We conclude that the function $h(\cdot,\cdot,\cdot)$ defined by \eqref{kernelBK} does provide the heat kernel. We stress that, for any $(w,\sigma,y), (w_0,\sigma_0,y_0)\in \R^n\times\R^k\times\R^{n_1}$ and $t>0$, $h((w,\sigma,y),(w_0,\sigma_0,y_0),t)$ is well-defined (and smooth) since, by arguing as in \eqref{iiii}-\eqref{chebound}, we have
\begin{align*}
&|h((w,\sigma,y),(w_0,\sigma_0,y_0),t)|\\
&\leq \int_{\R^k} e^{-\pi|\lambda|(\frac 12 \left\langle \sqrt{Q_1} w,w \right\rangle + \operatorname{tr} \sqrt{Q_1} t )}p_\la((w,y),(w_0,y_0),t)e^{\frac 12 \pi|\lambda| \left\langle \sqrt{Q_1} w_0,w_0 \right\rangle}d\lambda\\
&\leq (2\pi)^{-\frac{n+n_1}{2}} \int_{\R^k}\frac{e^{-\pi t|\lambda|\operatorname{tr} \sqrt{Q_1}}}{\sqrt{\det tK_\lambda(t)}} d\lambda <\infty,
\end{align*}
where in the last inequality we have used that $\operatorname{tr} \sqrt{Q_1} >0$ and the fact that $\left(\det tK_\lambda(t)\right)^{-1}$ grows at most polynomially with respect to $\lambda$ (see the explicit form of $tK_\lambda(t)$ in Example \ref{ex2} with $D_1=\pi|\lambda| \sqrt{Q_1}$). This finishes the proof of the theorem.

\end{proof}
}

If we set $n=n_1\in\N$ and $Q_1=B=I_n$, the PDE in \eqref{uno} becomes the hybrid equation highlighted in \eqref{one}. In this special situation, for $\lambda\in\R^k$ with $\lambda\neq 0$ and $t>0$, we have
$$
e^{t B_\la}= \begin{pmatrix} e^{- 2\pi t|\lambda|}I_n & O_{n\times n}\\ \frac{1-e^{-2\pi t |\lambda|}}{2\pi |\lambda|} I_n & I_n\end{pmatrix}
$$
and
$$
tK_\lambda(t)=e^{-2\pi t|\lambda|}\begin{pmatrix} \frac{\sinh(2\pi t|\lambda|)}{2\pi|\lambda|}I_n & \frac{\cosh(2\pi t|\lambda|)-1}{4\pi^2|\lambda|^2}I_n\\ \frac{\cosh(2\pi t|\lambda|)-1}{4\pi^2|\lambda|^2}I_n & \frac{e^{2\pi t|\lambda|}(2\pi t |\lambda|-\sinh(2\pi t |\lambda|))+(\cosh(2\pi t |\lambda|)-1)(e^{2\pi t|\lambda|}-1)}{8\pi^3|\lambda|^3}I_n\end{pmatrix},
$$
which yields the explicit formula \eqref{two} for the heat kernel.


\section{A class of heat kernels from conformal CR geometry}\label{S:CR}

In this section we construct the heat kernel of a class of hybrid evolution equations that play an important role in conformal CR geometry. In Section \ref{S:intro} we have already discussed the extension problem \eqref{fgmt} in the Heisenberg group $\Hn$ in the seminal work of Frank et al. \cite{FGMT}. More in general, we now consider a Lie group of Heisenberg type $\bG$ with logarithmic coordinates $(z,\sigma)$, where $z\in \R^m$ and $\sigma\in \R^k$ (see Section \ref{S:2}). If 
$\mathscr L = \Delta_z + \frac{|z|^2}{4} \Delta_\sigma  + \sum_{\ell = 1}^k \p_{\sigma_\ell} \Theta_\ell$,
denotes a horizontal Laplacian in $\bG$ (see \eqref{Lhtipo} below), then in this more general framework the parabolic counterpart of the  extension problem \eqref{fgmt} is as follows: given a function $u\in C^\infty_0(\bG\times \R)$, find a function $U\in C^\infty(\bG\times \R\times \R^+_y)$ such that
\begin{equation}\label{fgmtG}
\begin{cases}
\p_{yy} U  + \frac{1-2s}y \p_y U + \frac{y^2}4 \Delta_\sigma U + \mathscr L U  - \p_t U = 0,\ \ \ \ \ \ \ \ \ \ \ \ \ \ \ \ \ \text{in}\ \bG\times \R\times  \R^+_y,
\\
U(g,t,0) = u(g,t),\ \ \ \ \ \ \ \ \ \ \ \ \ \ \ \ \ \ \ \ \ \ \ \ \ \ \ \ \ \ \ \ \ \ \ \ \ \ \ \ \ \ \ \ \ \ \ \ \ \ \  (g,t)\in \bG\times \R.
\end{cases}
\end{equation} 
Our present objective is the computation of the heat kernel of the evolution PDE in \eqref{fgmtG}, i.e. of the equation defined in $\bG\times \R\times  \R^+_y$ as
\begin{equation}\label{defextLs}
\mathfrak L_{(s)} U -\partial_t U \overset{def}{=} \p_{yy} U  + \frac{1-2s}y \p_y U + \frac{y^2}4 \Delta_\sigma U + \mathscr L U  - \p_t U=0
\end{equation}
The following is our main result.

\begin{theorem}\label{T:ext0}
Let $\bG$ be a group of Heisenberg type. For every $0<s<1$ the heat kernel with pole at the origin of the operator $\mathfrak L_{(s)}-\partial_t$ in \eqref{defextLs} is given by
\begin{align}\label{parBfs00}
 \mathfrak q_{(s)}((z,\sigma),t,y) & =  \frac{2^k}{(4\pi t)^{\frac{m}2 +k +1-s}} \int_{\R^k} e^{- \frac it \langle \sigma,\la\rangle}   \left(\frac{|\la|}{\sinh |\la|}\right)^{\frac m2+1-s} e^{-\frac{|z|^2 +y^2}{4t}\frac{|\la|}{\tanh |\la|}} d\la.
\end{align}
\end{theorem}
We emphasise that although the heat kernel \eqref{parBfs00}, and its intertwined counterpart obtained by replacing $s$ with $-s$, have played a central role in our recent works \cite{GTfeel} and \cite{GTinter}, see also the earlier papers by Roncal and Thangavelu \cite{RTaim, RT}, their derivation was not given in the cited sources and it appears for the first time in the present paper.

As in the case of $\Hn$ (which constitutes the $k=1$ case of our treatment), a key observation here is that the PDE in \eqref{defextLs} is the restriction of the equation
\begin{equation}\label{hybG}
\Delta_w U+ \frac{|w|^2}{4} \Delta_\sigma U + \mathscr L U - \p_t U = 0
\end{equation}
to functions depending on the variable $y = |w|$, where $w$ belongs to the space with fractal dimension $\R^{2(1-s)}$. The link between \eqref{hybG} and the PDE in \eqref{defextLs} is readily seen by observing that, if $y = |w|$, then on a function $u(w) = \psi(y)$ we have $\Delta_w u = \p_{yy} \psi + \frac {1-2s}{y}\p_y \psi$. Another remark is that \eqref{hybG} is of hybrid type since the variable $\sigma$ appears in both equations 
\[
\Delta_w U+ \frac{|w|^2}{4} \Delta_\sigma U  - \p_t U = 0,
\]
and 
\[
\mathscr L U - \p_t U = 0,
\] 
see \eqref{Lhtipo} below for the expression of $\mathscr L$.
Also observe that, similarly to the situation of the operator in \eqref{uno} in Section \ref{S:met}, the PDE in \eqref{hybG} contains the limiting case in which the fractal dimension $n_1= 2(1-s) = 0$ of the variable $w$ vanishes, which is equivalent to letting $s \nearrow 1$. In such case the PDEs \eqref{defextLs}, \eqref{hybG} formally become 
\begin{equation}\label{heatG}
\mathscr L U - \p_t U= 0,
\end{equation}
the heat equation in $\bG$ associated with the horizontal Laplacian $\mathscr L$. For the latter the heat kernel is well-known and it is given by
\begin{align}\label{pipposemprepiupiccoloH}
& q((z,\sigma),(\zeta,\tau),t) = \frac{2^k}{\left(4\pi t\right)^{\frac m2+k}}\int_{\R^k} \left(\frac{|\la|}{\sinh |\la|}\right)^{\frac m2} e^{\frac it (\langle\tau -\sigma,\la\rangle + \frac 12\langle J(\la)\zeta,z\rangle)} e^{- \frac{|z-\zeta|^2}{4t} \frac{|\la|}{\tanh |\la|}}d\la.
\end{align}
In the special case of the Heisenberg group $\Hn$ one has $m=2n$, $k=1$, and \eqref{pipposemprepiupiccoloH} gives back the famous formula  independently found by Hulanicki \cite{Hu} and Gaveau \cite{Gav}. We mention here that in \cite{Fo} Folland proved the existence of the heat kernel in any Carnot group, but of course in such generality one does not have an explicit representation such as \eqref{pipposemprepiupiccoloH}. 

\begin{remark}\label{R:s=1}
The reader should note that if in \eqref{parBfs00} we formally set $s = 1$ and $y = 0$ we perfectly recover the Hulanicki-Gaveau formula \eqref{pipposemprepiupiccoloH} when the pole 
$(\zeta,\tau) = (0,0)$!
\end{remark}

Similarly to what we did in Section \ref{S:met}, in the present section we first provide in Theorem \ref{T:heat} a totally self-contained and elementary proof of the construction of the heat kernel for the limiting case \eqref{heatG}. We do this not just for groups of Heisenberg type, but in the more general framework of a Carnot group $\bG$ of step two. Of course the result per se is not new, as Cygan established it in \cite{Cy}, but our proof is. Although the relevant PDE \eqref{heatG} is not hybrid in the sense specified in the opening of this paper, the motivation for including here the construction of its heat kernel is twofold: (i) on one hand it allows to present our approach to Theorem \ref{T:ext0} in a significant, yet simplified setting; (ii) on the other hand we feel that our self-contained proof will be of interest to workers in analysis and PDEs who are not directly familiar with those important and deep tools, such as e.g. group representation theory, Laguerre calculus, complex Hamiltonians or a priori ansatz, which in one form or another have entered the previous related works such as \cite{Hu, Gav, Cy, BGG, Ran, Kli, B, CT, LP, BR, MM}. 


\subsection{The generalised harmonic oscillator with a complex drift}\label{S:hocomplex}

In what follows given a number $n\in \mathbb N$ we denote by
$S\in M_{n\times n}(\R)$ a skew-symmetric matrix, i.e., we assume $S^\star = - S$. We intend to solve the Cauchy problem for the following generalised harmonic oscillator with a complex drift
\begin{equation}\label{hocom}
\begin{cases}
\Delta_z  \tilde v  -  |S z|^2\  \tilde v + 2  i \langle S z, \nabla_z  \tilde v\rangle  - \p_t  \tilde v = 0,
\\
\tilde v(z,0) = \tilde v_0(z) \ \ \ \ \ \ \ \ \ \ \ \ \ z\in \R^n,\ t>0,
\end{cases}
\end{equation}
where $\tilde v_0$ is suitably chosen. 
We will need the following lemma that allows to eliminate the complex drift from \eqref{hocom}.
\begin{lemma}\label{L:youngandsmart}
Suppose that $v$ and $\tilde v$ are connected by the relation
\begin{equation}\label{drift}
v(z,t) = \tilde v(e^{-2 i t S} z,t).
\end{equation}
Then, $\tilde v$ is a solution to the PDE in \eqref{hocom} if and only if $v$ is a solution to the equation
\[
\Delta_z v  - |S z|^2 v  - \p_t v = 0.
\]
\end{lemma}
\begin{proof}
Let $\tilde v$ be a solution to the PDE in \eqref{hocom}. We note that by the skew-symmetry of $S$ we know that 
$$e^{-2 i t S^\star} e^{-2 i t S} = I_n.$$ 
Using this observation the reader can verify by a direct computation that 
$$
\Delta v(z,t) = \Delta \tilde v(e^{- 2 i t S} z,t).
$$
It is also clear from \eqref{drift} that
\[
\p_t  v(z,t) = - 2 i \langle S z,\nabla \tilde v(e^{-2 i t S} z,t)\rangle + \p_t \tilde v(e^{-2 i t S} z,t).
\]
Combining the latter two equations we easily reach the desired conclusion.

\end{proof}

Returning to the Cauchy problem \eqref{hocom} the following is the main result of this subsection.

\begin{proposition}\label{P:Pippetto}
Let
\begin{equation}\label{wowQ}
\mathscr Q(z,\zeta,t) =\frac{e^{i\langle Sz,\zeta\rangle}}{(4\pi t)^{\frac n2}} \left(\det j\left(2t\sqrt{S^\star S}\right)\right)^{\frac{1}{2}} e^{-\frac{1}{4t} \langle j\left(2t\sqrt{S^\star S}\right) \cosh\left(2t\sqrt{S^\star S}\right) (z-\zeta), z-\zeta\rangle}
\end{equation}  
Then, for any $\tilde v_0\in C(\Rn)$ such that
$$
e^{-\frac{1}{2}\left\langle \sqrt{S^\star S} z , z \right\rangle}\tilde v_0(z) \in L^\infty(\Rn),
$$
the function
\begin{equation}\label{mellettoo}
\tilde v(z,t) = \int_{\Rn} \mathscr Q(z,\zeta,t) \tilde v_0(\zeta) d\zeta
\end{equation}
solves \eqref{hocom}.
\end{proposition}
\begin{proof}

It is clear that using Lemma \ref{L:youngandsmart} the Cauchy problem \eqref{hocom} is transformed into 
\begin{equation}\label{hocom2}
\begin{cases}
\Delta_z   v  -  |S z|^2\  v   - \p_t   v = 0,
\\
v(z,0) = \tilde v_0(z) \ \ \ \ \ \ \ \ \ \ \ \ \ z\in \R^n,\ t>0,
\end{cases}
\end{equation}
for the function $v$ defined by \eqref{drift}. If we now define $D = \sqrt{S^\star S} = \sqrt{-S^2}$, then clearly $D\ge 0$,   $D^\star = D$, and $|Sz|^2 = |Dz|^2$. We can thus re-write \eqref{hocom2} as follows 
\[
\begin{cases}
\Delta_z   v  -  |D z|^2\  v   - \p_t   v = 0,
\\
v(z,0) = \tilde v_0(z) \ \ \ \ \ \ \ \ \ \ \ \ \ z\in \R^n,\ t>0.
\end{cases}
\]
According to Proposition \ref{P:Pippo} the function
\[
v(z,t) = \int_{\Rn} \mathscr M(z,\zeta,t) \tilde v_0(\zeta) d\zeta,
\]
with $D =  \sqrt{-S^2}$ and $\mathscr M(z,\zeta,t)$ as in \eqref{Pippo}, solves the latter problem. Undoing \eqref{drift} we have proved that the function defined by 
\begin{equation}\label{melletto}
\tilde v(z,t) = \int_{\Rn} \mathscr M(e^{2 i t S} z,\zeta,t) \tilde v_0(\zeta) d\zeta
\end{equation}
solves \eqref{hocom}.
We next want to further simplify the expression \eqref{melletto}. Keeping in mind the explicit expression \eqref{Pippo} for $\mathscr M$, we  have
\begin{align}\label{Pippotto}
 \mathscr M(e^{2 i t S} z,\zeta,t) & = (4\pi t)^{-\frac n2} \sqrt{\det j(2tD)} \exp\bigg\{-\frac{1}{4t} \bigg(\langle j(2tD) \cosh(2tD) e^{2 i t S} z, e^{2 i t S} z\rangle
\\
& + \langle j(2tD) \cosh(2tD) \zeta,\zeta\rangle - 2 \langle j(2tD)\ e^{2 i t S} z,\zeta \rangle\bigg)\bigg\},
\notag
\end{align}
We now observe that, since $S$ commutes with any even analytic function of $D = \sqrt{-S^2}$, we have in particular
\begin{equation}\label{ollaoppi}
\langle j(2tD) \cosh(2 t D) e^{2it S}z,e^{2i tS}z\rangle = \langle j(2tD) \cosh(2 t D) z,z\rangle, 
\end{equation}
as well as
\begin{align}\label{eiS}
e^{2i tS}&=\sum_{k=0}^{\infty}\frac{(2it S)^{2k}}{(2k)!} + \sum_{k=0}^{\infty}\frac{(2i tS)^{2k+1}}{(2k+1)!}\\
&=\sum_{k=0}^{\infty}\frac{(-4t^2 S^2)^{k}}{(2k)!} +2it S \sum_{k=0}^{\infty}\frac{(-4t^2 S^2)^{k}}{(2k+1)!}\nonumber\\
&=\sum_{k=0}^{\infty}\frac{(2tD)^{2k}}{(2k)!} +2i t S \sum_{k=0}^{\infty}\frac{(2tD)^{2k}}{(2k+1)!}\nonumber\\
&=\cosh(2tD) +2i t \left(j(2tD)\right)^{-1} S.\nonumber
\end{align}
We thus find from \eqref{eiS} 
\begin{equation}\label{ollaoppi2}
- 2 \langle j(2tD)\ e^{2 i t S} z,\zeta \rangle = - 2 \langle j(2tD)\ \cosh(2tD) z,\zeta \rangle - 4 i t \langle Sz,\zeta\rangle.
\end{equation}
Replacing now \eqref{ollaoppi} and \eqref{ollaoppi2} in \eqref{Pippotto}, we finally obtain
\begin{align*}
 \mathscr M(e^{2 i t S} z,\zeta,t) & = (4\pi t)^{-\frac n2} \sqrt{\det j(2tD)} e^{i\langle Sz,\zeta\rangle} \exp\bigg\{-\frac{1}{4t} \bigg(\langle j(2tD) \cosh(2tD) z, z\rangle
\\
& + \langle j(2tD) \cosh(2tD) \zeta,\zeta\rangle - 2 \langle j(2tD)\ \cosh(2tD) z,\zeta \rangle\bigg)\bigg\}\\
&= \frac{e^{i\langle Sz,\zeta\rangle}}{(4\pi t)^{\frac n2}} \sqrt{\det j(2tD)} \exp\bigg\{-\frac{1}{4t}\langle j(2tD) \cosh(2tD) (z-\zeta), z-\zeta\rangle\bigg\}\\
&=\mathscr Q(z,\zeta,t).
\end{align*}
This concludes the proof.

\end{proof}

It is clear from the previous proof that the kernel $\mathscr Q(z,\zeta,t)$ is equal to $\mathscr M(e^{2 i t S} z,\zeta,t)$ with $\mathscr M$ given by \eqref{Pippo} and with $D=\sqrt{S^\star S}$. Using the commutation property in \eqref{ollaoppi} we then deduce from \eqref{intpippo} that
\begin{align}\label{intpippetto}
&\int_{\R^n} \mathscr Q(z,\zeta,t) d\zeta\\
&= \frac{1}{\sqrt{\det \cosh 2t\sqrt{S^\star S}}}\exp\bigg\{-\frac{1}{4t}\langle j(2t\sqrt{S^\star S}) \left(\cosh 2t\sqrt{S^\star S} - \cosh^{-1} 2t\sqrt{S^\star S}\right)\ z, z\rangle \bigg\}.\notag
\end{align}
This equation will be used in the proof of Theorem \ref{T:heat} below. 


\subsection{Gaveau, Hulanicki and Cygan met Ornstein and Uhlenbeck}\label{S:2}

Henceforth, we denote by $\bG$ a Carnot group of step two and we let $\bg = V_1 \oplus V_2$ indicate its Lie algebra, with inner product $\langle \cdot,\cdot\rangle$. Recall that the step two assumption means that $[V_1,V_1] = V_2$ and that $[V_1,V_2] = \{0\}$. We let $m = \operatorname{dim}(V_1)$, $k = \operatorname{dim}(V_2)$, and we fix orthonormal basis $\{e_1,...,e_{m}\}$ and $\{\ve_1,...,\ve_k\}$ for $V_1$ and $V_2$ respectively. For points $z\in V_1$ and $\sigma\in V_2$ we will use either one of the representations $z = \sum_{j=1}^{m} z_j e_j$, $\sigma = \sum_{\ell=1}^k \sigma_\ell \ve_\ell$, or also $z = (z_1,...,z_{m})$, $\sigma = (\sigma_1,...,\sigma_k)$. Whenever convenient, we routinely identify $V_1 \cong \Rm$ and $V_2 \cong \R^k$ and the points $g, g'\in \bG$ with their logarithmic coordinates $(z,\sigma), (\zeta,\tau)\in \R^{m+k}$ respectively. Recall that the Kaplan mapping $J: V_2 \to \operatorname{End}(V_1)$ is defined by 
\begin{equation}\label{kap}
\langle J(\sigma)z,\zeta\rangle = \langle [z,\zeta],\sigma\rangle = - \langle J(\sigma)\zeta,z\rangle.
\end{equation}
Clearly, $J(\sigma)^\star = - J(\sigma)$, and one has $\langle J(\sigma)z,z\rangle = 0$. Moreover, by the bracket generating assumption $[V_1,V_1]=V_2$, we know that the map $J$ is injective. 
Via the Baker-Campbell-Hausdorff formula 
\begin{equation}\label{bch}
\exp(z+\sigma)  \exp(\zeta+\tau) = \exp{\left(z + \zeta + \sigma +\tau + \frac{1}{2}
[z,\zeta]\right)}
\end{equation}
the map \eqref{kap} identifies the non-Abelian multiplication law in $\bG$
$$g\circ g' = (z +\zeta,\sigma + \tau + \frac 12 \sum_{\ell=1}^{k} \langle J(\ve_\ell)z,\zeta\rangle\ve_\ell).$$
For future use we observe that
\begin{equation}\label{gl2}
(g')^{-1}\circ g = (z -\zeta,\sigma - \tau + \frac 12 \sum_{\ell=1}^{k} \langle J(\ve_\ell)z,\zeta\rangle\ve_\ell).
\end{equation}
The map \eqref{kap} also induces a  complex geometry in the Lie group. This becomes particularly transparent in the special case of groups of Heisenberg type for which $J(\sigma)^2 = - |\sigma|^2 I_{V_1}$, see the seminal works \cite{Ka} and \cite{CDKR}. In general,  for the mapping 
\begin{equation}\label{A}
A(\sigma) \overset{def}{=} J^\star(\sigma) J(\sigma) = - J(\sigma)^2,
\end{equation}
we have $\langle A(\sigma) z,z\rangle = |J(\la)z|^2\ge 0$, and therefore it defines a symmetric nonnegative element of $\operatorname{End}(V_1)$ for every $\sigma \in V_2$. Consequently, the matrix $\sqrt{A(\sigma)}$ is well-defined.
\begin{remark}\label{remexp}
Notice that $\operatorname{Ker} A(\la) = \operatorname{Ker} J(\la)$ can have positive dimension. Nevertheless, the injectivity of $J$ ensures that $A(\la)$ is not the null endomorphism for every $\la\neq 0$. Moreover, being $J(\la)$ skew-symmetric, the dimension of the range of $A(\la)$ has to be at least two. Since the linearity of $J$ allows us to write $\sqrt{A(\la)}=|\la|\sqrt{A(\la/|\la|)}$, one can deduce that there exists $k_0>0$ such that $\sqrt{A(\la)}$ has at least two eigenvalues bigger than $k_0|\la|$. This implies that  
\begin{equation}\label{expdecaybound}
\det j(\sqrt{A(\la)})\leq \left(j(k_0|\la|)\right)^2 =\left(\frac{2k_0|\la|}{e^{k_0|\la|}-e^{-k_0|\la|}}\right)^2.
\end{equation} 
\end{remark}
If for $j=1,...,m$ we define left-invariant vector fields by the Lie rule $X_j u(g) = \frac{d}{ds} u(g \circ\exp s e_j)\bigg|_{s=0}$,
then by \eqref{bch} one obtains in the logarithmic coordinates $(z,\sigma)$ 
\begin{equation}\label{Xi}
X_j  = \p_{z_j} + \frac 12 \sum_{\ell=1}^k \langle J(\ve_\ell)z,e_j\rangle \partial_{\sigma_\ell} = \p_{z_j} + \frac 12 \sum_{\ell=1}^k \sum_{i=1}^m  z_i \langle J(\ve_\ell)e_i,e_j\rangle \partial_{\sigma_\ell}.
\end{equation}
Although we will not make an explicit use of this fact, we note that \eqref{Xi} implies the following commutation relation 
\[
[X_i,X_{j}] = \sum_{\ell=1}^k \langle J(\ve_\ell) e_i,e_j\rangle \p_{\sigma_\ell}.
\]
Given a function $f\in C^1$ we will indicate by $\nabla_H f = (X_1f,...,X_{m} f)$ its horizontal gradient, and set $|\nabla_H f| = (\sum_{j=1}^{m} (X_j f)^2)^{1/2}$. The horizontal Laplacian generated by the orthonormal basis $\{e_1,...,e_{m}\}$ of $V_1$ is the second-order differential operator on $\bG$ defined by 
$$\mathscr L f = \sum_{j=1}^{m} X_j^2 f,$$
where $X_1,...,X_{m}$ are given by \eqref{Xi}. A computation gives
\begin{equation}\label{sl}
\mathscr L = \Delta_z + \frac 14 \sum_{\ell,\ell' = 1}^k \langle J(\ve_\ell)z,J(\ve_{\ell'})z\rangle \p_{\sigma_\ell}\p_{\sigma_{\ell'}} + \sum_{\ell = 1}^k \Theta_\ell \p_{\sigma_\ell},
\end{equation}
where $\Delta_z$ represents the standard Laplacian in the variable $z = (z_1,...,z_m)$, and
\begin{equation}\label{thetaell0}
\Theta_\ell = \sum_{s=1}^m \langle J(\ve_\ell)z,e_s\rangle \p_{z_s}.
\end{equation}

For $f\in \mathscr S$ (recall that we are thinking of $\bG \cong \R^m\times \R^k$) we now consider the Cauchy problem
\begin{equation}\label{cp}
\begin{cases}
\mathscr L u - \p_t u = 0\ \ \ \ \ \ \ \ \text{in}\ \bG\times (0,\infty),
\\
u(g,0) = f(g) \ \ \ \ \ \ \ \ \ \ \ \ \ g\in \bG.
\end{cases}
\end{equation}

The aim of this section is to provide a new simple proof of the following classical result. 

\begin{theorem}[Gaveau-Hulanicki-Cygan]\label{T:heat}
The heat kernel in $\bG\times (0,\infty)$ is given by the formula
\begin{align}\label{ournucleo}
& q(g,g',t) = 2^k (4\pi t)^{-(\frac m2 + k)}  \int_{\R^k} e^{\frac{i}{t}\left(\langle \tau-\sigma,\la\rangle +\frac{1}{2}\langle J(\la)\zeta,z\rangle \right)} \left(\det j(\sqrt{A(\la)})\right)^{1/2} 
\\
& \times \exp\bigg\{-\frac{1}{4t}\langle j(\sqrt{A(\la)}) \cosh \sqrt{A(\la)} (z-\zeta),z-\zeta\rangle \bigg\}
d\la,\notag
\end{align}
where $g=(z,\sigma)$ and $g' = (\zeta,\tau)$.
\end{theorem}
It is appropriate to mention that although the integral representations in the above cited literature appear in different forms with respect to \eqref{ournucleo}, they are in fact equivalent to it. The advantage of our presentation is that it is particularly transparent and is easily applicable for instance to delicate questions in geometric measure theory treated in our work \cite{GTbbmd}.
\begin{proof}[Proof of Theorem \ref{T:heat}]
Fix $f\in \mathscr S$. To solve \eqref{cp} we start from the expression \eqref{sl} of the horizontal Laplacian, keeping also \eqref{thetaell0} in mind. We identify $\bG$ with $\R^m\times \R^k$ and, denoting for $\la\in \R^k$
$$
\hat u(z,\la,t) = \int_{\R^k} e^{-2\pi i\langle\la,\sigma\rangle} u(z,\sigma,t) d\sigma,
$$
the partial Fourier transform of $u$ with respect to the central variable $\sigma\in \R^k$, we obtain from \eqref{sl}
\begin{equation}\label{cp2}
\begin{cases}
\Delta_z \hat u  - \pi^2 \sum_{\ell,\ell' = 1}^k \la_\ell \la_{\ell'}\langle J(\ve_\ell)z,J(\ve_{\ell'})z\rangle  \hat u + 2 \pi i \sum_{\ell = 1}^k \la_\ell \Theta_\ell  \hat u - \p_t \hat u = 0,
\\
\hat u(z,\la,0) = \hat f(z,\la) \ \ \ \ \ \ \ \ \ \ \ \ \ z\in \R^m,\ \la \in \R^k,\ t>0.
\end{cases}
\end{equation}
Now, it is not difficult to recognise that
\begin{equation}\label{mush1}
\sum_{\ell,\ell' = 1}^k \la_\ell \la_{\ell'}\langle J(\ve_\ell)z,J(\ve_{\ell'})z\rangle = |J(\la)z|^2.
\end{equation}
Furthermore, we obtain from \eqref{thetaell0}
\begin{equation}\label{mush2}
\sum_{\ell = 1}^k \la_\ell \Theta_\ell  \hat u = \sum_{\ell = 1}^k \la_\ell  \sum_{s=1}^m \langle J(\ve_\ell)z,e_s\rangle \p_{z_s} \hat u = \sum_{s=1}^m \langle J(\la)z,e_s\rangle \p_{z_s} \hat u = \langle J(\la) z,\nabla_z \hat u\rangle.
\end{equation}
Using the identities \eqref{mush1}, \eqref{mush2} we can thus write \eqref{cp2} in the form
\begin{equation}\label{cp3}
\begin{cases}
\Delta_z \tilde v  - \pi^2 |J(\la)z|^2\ \tilde v + 2 \pi i \langle J(\la) z, \nabla_z \tilde v\rangle  - \p_t \tilde v = 0,
\\
\tilde v(z,0) = \hat f(z,\la) \ \ \ \ \ \ \ \ \ \ \ \ \ z\in \R^m,\ t>0,
\end{cases}
\end{equation}
where for $\la\in \R^k$ fixed, we have let $\tilde v(z,t) = \hat u(z,\la,t)$. To solve problem \eqref{cp3} we now apply Proposition \ref{P:Pippetto} with the choice of the skew-symmetric matrix $S = \pi J(\la)$, so that the symmetric matrix $D = \sqrt{-S^2}$ is presently given by $D = \pi \sqrt{A(\la)} = \pi \sqrt{-J(\la)^2}$ (we warn the reader that the role of the dimension $n$ in Proposition \ref{P:Pippetto} is now taken by the dimension $m$ of the first layer $V_1$). Formula \eqref{mellettoo} allows to conclude that
$$
\hat u(z,\la,t) = \int_{\Rm} \mathscr Q( z,\zeta,t) \hat f(\zeta,\la) d\zeta
$$
with the kernel $\mathscr Q$ as in \eqref{wowQ}. We can now argue as in the proof of Theorem \ref{T:BGpar}. First, by exploiting \eqref{intpippetto}, we can ensure that $\hat u(z,\la,t)\to \hat f(z,\la)$ in $L^1(\R^k,d\lambda)$ as $t\to 0^+$. Then, we can conclude that the function
$$
u(z,\sigma,t) = \int_{\Rm} \int_{\R^k}  \int_{\R^k} e^{2\pi i\langle \sigma-\tau,\la\rangle} \mathscr Q( z,\zeta,t)  f(\zeta,\tau) d\la   d\zeta d\tau
$$
solves the Cauchy problem \eqref{cp}. The desired heat kernel on the group $\bG$ is thus given by
\begin{align*}
 q((z,\sigma),(\zeta,\tau),t) &= \int_{\R^k} e^{2\pi i\langle \sigma-\tau,\la\rangle} \mathscr Q(z,\zeta,t) d\lambda\\
&=(4\pi t)^{-\frac m2} \int_{\R^k} e^{2\pi i\langle \sigma-\tau,\la\rangle} e^{i\pi \langle J(\lambda)z,\zeta\rangle} \left(\det j(2\pi t \sqrt{A(\la)})\right)^{1/2} \times \\
&\times \exp\bigg\{-\frac{1}{4t}\langle j(2\pi t\sqrt{A(\la)}) \cosh \left(2\pi t\sqrt{A(\la)}\right) (z-\zeta),z-\zeta\rangle \bigg\} d\lambda
\end{align*}
where we have used the expression \eqref{wowQ}. Making the change of variable $(-2\pi t \la) \to \la$ and exploiting the linearity and the skew-symmetry of the mapping $J(\cdot)$, we finally obtain
\begin{align}\label{pipposemprepiupiccolo}
& q((z,\sigma),(\zeta,\tau),t) = 2^k (4\pi t)^{-(\frac m2 + k)}  \int_{\R^k} e^{\frac{i}{t}\left(\langle \tau-\sigma,\la\rangle +\frac{1}{2}\langle J(\la)\zeta,z\rangle \right)} \left(\det j(\sqrt{A(\la)})\right)^{1/2} 
\\
& \times \exp\bigg\{-\frac{1}{4t}\langle j(\sqrt{A(\la)}) \cosh \sqrt{A(\la)} (z-\zeta),z-\zeta\rangle \bigg\}
d\la.\notag
\end{align}
From \eqref{expdecaybound} it is easy to verify that the integral in \eqref{pipposemprepiupiccolo} is finite, and the kernel $q((z,\sigma),(\zeta,\tau),t)$ is a smooth function. This completes the proof of Theorem \ref{T:heat}.

\end{proof}

For the sake of completeness we close this brief subsection by showing how to recover the Gaveau-Hulanicki formula \eqref{pipposemprepiupiccoloH} from Theorem \ref{T:heat}. We recall that a Carnot group of step two $\bG$ is said of Heisenberg type if for every $\la\in V_2$ the mapping $A(\la)$ in \eqref{A} satisfies $A(\la) = - J(\la)^2 = |\la|^2 I_m$. Therefore, $\sqrt{A(\la)} = |\la| I_m$, and we have 
$$j(\sqrt{A(\la)}) = j(|\la| I_m) = j(|\la|) I_m,$$
and also
\[
 j(\sqrt{A(\la)}) \cosh \sqrt{A(\la)}  = j(|\la|) \cosh |\la| I_m  = \frac{|\la|}{\tanh |\la|} I_m.
\]
We thus obtain from \eqref{ournucleo}
\begin{align*}
& q((z,\sigma),(\zeta,\tau),t) = \frac{2^k}{\left(4\pi t\right)^{\frac m2+k}}\int_{\R^k} \left(\frac{|\la|}{\sinh |\la|}\right)^{\frac m2} e^{\frac it (\langle\tau -\sigma,\la\rangle + \frac 12\langle J(\la)\zeta,z\rangle)} e^{- \frac{|z-\zeta|^2}{4t} \frac{|\la|}{\tanh |\la|}}d\la
\end{align*}
which coincides with the expression recalled in \eqref{pipposemprepiupiccoloH}.


\subsection{Proof of Theorem \ref{T:ext0}}\label{S:ext}

In this section we finally turn to the proof of Theorem \ref{T:ext0}. As we have already observed, we begin from the crucial observation that, in a group of Heisenberg type $\bG$, the relevant heat equation associated with \eqref{defextLs} is \eqref{hybG}, with the variable $w$ running in the space with fractal dimension $\R^{2(1-s)}$. Henceforth, to continue the analysis we proceed formally and treat the number $2(1-s)$ as if it were an integer. In this way we will arrive to a specific candidate for a heat kernel in the form \eqref{parBfs00}. Only at that point we will rigorously justify our computations and complete the proof. 

We begin by observing that the assumption $J(\la)^2 = - |\la|^2 I_m$ implies in particular that $\langle J(\ve_\ell)z,J(\ve_{\ell'}) z\rangle = |z|^2 \delta_{\ell \ell'}$ for every $z\in V_1$ and every $\ell, \ell'\in \{1,...,k\}$ (see, e.g, \cite[Prop. 2.9]{Gems}). Inserting this information in \eqref{sl} we conclude that in a group of Heisenberg type the horizontal Laplacian is given by 
\begin{equation}\label{Lhtipo}
\mathscr L = \Delta_z + \frac{|z|^2}{4} \Delta_\sigma + \sum_{\ell = 1}^k \Theta_\ell \p_{\sigma_\ell}.
\end{equation}
Combining \eqref{hybG} and \eqref{Lhtipo}, it is clear that we presently want to solve the Cauchy problem in $\Rm\times \R^k\times\R^{2(1-s)} \times (0,\infty)$
\begin{equation}\label{cauchyone}
\begin{cases}
\Delta_z U + \Delta_w U+ \frac{|z|^2 + |w|^2}{4} \Delta_\sigma U + \sum_{\ell = 1}^k \Theta_\ell \p_{\sigma_\ell} U - \p_t U = 0,
\\
U(z,\sigma,w,0) = u_0(z,\sigma,w),
\end{cases} 
\end{equation}
where the function $u_0$ is assigned in $\mathscr S \left(\Rm\times \R^k\times\R^{2(1-s)}\right)$. Proceeding as in the Section \ref{S:2} we now take a partial Fourier transform of the problem \eqref{cauchyone} with respect to the variable $\sigma\in \R^k$. Denoting by $
\hat U(z,\la,w,t) = \int_{\R^k} e^{-2\pi i\langle\la,\sigma\rangle} U(z,\sigma,w,t) d\sigma$, and using the identity \eqref{mush2}, we obtain from \eqref{cauchyone}
$$\begin{cases}
(\Delta_z + \Delta_w)\hat U  - \pi^2 |\la|^2  (|z|^2+|w|^2)  \hat U + 2 \pi i \langle J(\la) z, \nabla_z \hat U\rangle - \p_t \hat U = 0,
\\
\hat U(z,\la,w,0) = \hat u_0(z,\la,w) \ \ \ \ \ \ \ \ \ \ \ \ \ (z,\la,w)\in \R^m\times \R^k \times\R^{2(1-s)},\ t>0.
\end{cases}$$
At this point we make critical use of the elimination of the complex drift transformation introduced in Lemma \ref{L:youngandsmart}. For $\la\in \R^k$, we define
$$
V(z,w,t) = \hat U(e^{-2\pi i t J(\la)} z, \la, w,t).
$$
We leave it to the reader to verify that, similarly to Lemma \ref{L:youngandsmart}, the function $V$ now solves the problem
$$\begin{cases}
(\Delta_z + \Delta_w) V  - \pi^2 |\la|^2 (|z|^2+|w|^2) V  - \p_t V = 0,
\\
V(z,w,0) =  \tilde U(z,w,0).
\end{cases}$$
According to Proposition \ref{P:Pippo}, if we denote by $\mathscr M$ the kernel in \eqref{Pippo} with the choice of $D = \pi |\la| I_{m+2(1-s)}$, we deduce that $V$ is given by
$$
V(z,w,t)=\int_{\R^m\times\R^{2(1-s)}} \mathscr M ((z,w),(\zeta,w'),t) \tilde U(\zeta,w',0) d\zeta dw'
$$
which we can rewrite as
$$
\hat U(z,\la,w,t)=\int_{\R^m\times\R^{2(1-s)}} \mathscr M ((e^{2\pi i t J(\la)}z,w),(\zeta,w'),t) \hat u_0(\zeta,\la,w') d\zeta dw'.
$$
By taking the inverse Fourier transform we then have 
\begin{align}\label{solqq}
&U(z,\sigma,w,t)= \\
&=\int_{\R^m\times\R^k\times\R^{2(1-s)}} \int_{\R^k} e^{2\pi i \left\langle \sigma-\tau,\lambda\right\rangle} \mathscr M ((e^{2\pi i t J(\la)}z,w),(\zeta,w'),t) u_0(\zeta,\tau,w') d\lambda d\zeta d\tau dw'.\notag
\end{align}
From \eqref{solqq} we can extract the following heat kernel for the problem \eqref{cauchyone}
$$
Q((z,\sigma,w),(\zeta,  \tau, w'),t)=\int_{\R^k} e^{2\pi i \left\langle \sigma-\tau,\lambda\right\rangle} \mathscr M ((e^{2\pi i t J(\la)}z,w),(\zeta,w'),t) d\lambda.
$$
By arguing as in \eqref{Pippotto}-\eqref{ollaoppi2} (where we have deduced the expression of $\mathscr Q$) and similarly to \eqref{pipposemprepiupiccolo}, we obtain from the explicit expression of $\mathscr M$ that
\begin{align}\label{kerneldelwudoppio}
Q((z,\sigma,w),(\zeta,  \tau, w'),t)&=\frac{2^k}{(4\pi t)^{\frac{m}2 +k +1-s}} \int_{\R^k} e^{\frac it (\langle\tau -\sigma,\la\rangle + \frac 12\langle J(\la)\zeta,z\rangle)}   \left(\frac{|\la|}{\sinh |\la|}\right)^{\frac m2+1-s}\times \\
&\times e^{-\frac{1}{4t}\left( \frac{|\la|}{\tanh |\la|}\left(|z-\zeta|^2 +|w|^2+|w'|^2\right)-2\frac{|\la|}{\sinh |\la|}\left\langle w,w'\right\rangle\right) } d\la.\notag
\end{align}
If we set $w'=0$ into \eqref{kerneldelwudoppio}, and we keep \eqref{gl2} in mind, as well as the definition of $\mathfrak{q}_{(s)}(\cdot,\cdot,\cdot)$ in \eqref{parBfs00}, we easily realize that
$$
Q((z,\sigma,w),(\zeta,  \tau, 0),t)=Q(((\zeta,\tau)^{-1}\circ (z,\sigma),w),((0,0),0),t)=\mathfrak{q}_{(s)}((\zeta,\tau)^{-1}\circ (z,\sigma),t,|w|). 
$$
The previous identities, together with the already discussed relation between the equations \eqref{hybG} and \eqref{defextLs}, make the function
$$
\mathfrak{q}_{(s)}((z,\sigma),t,y)
$$
the candidate heat kernel of $\mathfrak L_{(s)}-\partial_t$ with pole at the origin for any value of the fractional parameter $s\in (0,1)$. We now rigorously prove this fact, thus establishing Theorem \ref{T:ext0}.

{\allowdisplaybreaks
On the one hand, a straightforward computation shows that $\mathfrak{q}_{(s)}((z,\sigma),t,y)$ does solve the equation \eqref{defextLs} in $\{t>0\}$. We are thus left with understanding in which sense the kernel $\mathfrak{q}_{(s)}((z,\sigma),t,y)$ approaches the delta-function $\delta_{(0,0,0)}$ as $t\to 0^+$. With this in mind we introduce the measure
$$
d \mu \overset{def}{=} \frac{2\pi^{1-s}}{\G(1-s)} y^{1-2s} dy d\sigma dz.
$$
This is a natural measure for the operator $\mathfrak L_{(s)}$ since it is easy to check that $\mathfrak L_{(s)}$ is symmetric (i.e. formally self-adjoint) with respect to $d\mu$. Moreover, the renormalising constant $\frac{2\pi^{1-s}}{\G(1-s)}$ has been chosen in such a way that the following lemma holds true.

\begin{lemma}\label{massuno}
For every $t>0$ we have
$$
\int_{\R^m\times\R^k\times (0,\infty)} \mathfrak{q}_{(s)}((z,\sigma),t,y) d\mu =1.
$$
\end{lemma}
\begin{proof}
For $M>0$ denote
$$
K_M=\left\{(z,\sigma,y)\in\R^m\times\R^k\times (0,\infty)\,:\, |\sigma|\leq M\right\}.
$$
We want to show that
\begin{equation}\label{todo}
\exists \,\underset{M\to +\infty}{\lim} \int_{K_M} \mathfrak{q}_{(s)}((z,\sigma),t,y) d\mu =1.
\end{equation}
It is easy to see that, for any fixed $t>0$,
\begin{align*}
&\int_{K_M} \left|\mathfrak{q}_{(s)}((z,\sigma),t,y)\right| d\mu \\
&\leq \frac{2^k}{(4\pi t)^{\frac{m}2 +k +1-s}}\int_{K_M}\left(\int_{\R^k}\left(\frac{|\la|}{\sinh |\la|}\right)^{\frac m2+1-s} d\la\right) e^{-\frac{|z|^2 +y^2}{4t}} d\mu<\infty.
\end{align*}
By Fubini's theorem we are then free to choose the order of integration. With this in mind, we notice that
\begin{align*}
&\int_{\R^m}\int_0^\infty \mathfrak{q}_{(s)}((z,\sigma),t,y)\left(\frac{2\pi^{1-s}}{\G(1-s)} y^{1-2s}\right)dydz\\
&=\frac{2^k}{(4\pi t)^{k}} \int_{\R^k} e^{- \frac it \langle \sigma,\la\rangle}   \left(\frac{|\la|}{\sinh |\la|}\right)^{\frac m2+1-s} \left(\frac{\tanh |\la|}{|\la|}\right)^{\frac m2+1-s} d\la\\
&=\int_{\R^k} e^{- 2\pi i \langle \sigma,\la\rangle}   \left(\frac{1}{\cosh 2\pi t|\la|}\right)^{\frac m2+1-s} d\la,
\end{align*}
which implies
\begin{align*}
\int_{K_M} \mathfrak{q}_{(s)}((z,\sigma),t,y) d\mu = \int_{\{\sigma\in\R^k\,:\,|\sigma|\leq M\}}\int_{\R^k} e^{- 2\pi i \langle \sigma,\la\rangle}   \left(\frac{1}{\cosh 2\pi t|\la|}\right)^{\frac m2+1-s} d\la d\sigma.
\end{align*}
We now stress that, being the function $f_t(\la):=\left(\frac{1}{\cosh 2\pi t |\la|}\right)^{\frac m2+1-s}$ in the Schwartz class $\mathscr{S}\left(\R^k\right)$, its Fourier transform $\hat{f}_t$ is still in $\mathscr{S}\left(\R^k\right)$ and in particular in $L^1\left(\R^k\right)$. Hence we have
$$
\exists \,\underset{M\to +\infty}{\lim}\int_{K_M} \mathfrak{q}_{(s)}((z,\sigma),t,y) d\mu= \underset{M\to +\infty}{\lim} \int_{\{\,|\sigma|\leq M\}} \hat{f}_t(\sigma) d\sigma= \int_{\R^k} \hat{f}_t(\sigma) d\sigma=1,
$$
where in the last step we have applied the inversion theorem for the Fourier transform and the fact that $f_t(0)=1$. This concludes the proof of \eqref{todo}.

\end{proof}

In order to complete the proof of Theorem \ref{T:ext0} we need to show the validity of the following limiting relation:
\begin{align}\label{finally}
\forall\, \phi \in \mathscr{S}\left(\R^m\times\R^k\times [0,\infty)\right)\,\,\,\mbox{ we have }&\notag\\
\underset{t\to 0^+}{\lim} \int_{\R^m\times\R^k\times (0,\infty)} \mathfrak{q}_{(s)}((z,\sigma),t,y) \phi(z,\sigma,y) d\mu = \,& \phi(0,0,0).
\end{align}
For any fixed $\phi$ as in \eqref{finally}, let us introduce the functions $g_0,\, g_t:\R^k\to \R$ defined, for any $t>0$, by
$$
g_0(\lambda)=\hat{\phi}(0,\lambda,0)
$$
and
\begin{align*}
g_t(\lambda)=&\frac{2\pi^{1-s}}{\G(1-s)(4\pi t)^{\frac{m}2 +1-s}} \left(\frac{2\pi t|\la|}{\sinh 2\pi t|\la|}\right)^{\frac m2+1-s}\times\\
&\times \int_{\R^m\times(0,\infty)} e^{-\frac{|z|^2 +y^2}{4t}\frac{2\pi t|\la|}{\tanh 2\pi t|\la|}} \hat{\phi}(z,\lambda,y) y^{1-2s} dy dz.
\end{align*}
We notice that $g_0,\, g_t\in \mathscr{S}\left(\R^k\right)$. We want to show that
\begin{equation}\label{mainclaim}
g_t \to g_0 \quad \mbox{ in }L^1\left(\R^k\right)\quad\mbox{ as }t\to 0^+.
\end{equation}
In order to prove \eqref{mainclaim} we first notice that, by exploiting the change of variables 
\[
z\mapsto\xi=\frac{z}{\sqrt{4t}}\sqrt{\frac{2\pi t|\la|}{\tanh 2\pi t|\la|}}\ \ \  \text{and}\ \ \  y\mapsto\eta=\frac{y^2}{4t}\frac{2\pi t|\la|}{\tanh 2\pi t|\la|},
\]
we have
\begin{align*}
&g_t(\lambda)=\frac{1}{\G(1-s)\pi^{\frac{m}2}} \left(\frac{1}{\cosh 2\pi t|\la|}\right)^{\frac m2+1-s} \times\\
&\times\int_{\R^m\times(0,\infty)} \eta^{-s}e^{-|\xi|^2-\eta} \hat{\phi}\left(\xi\sqrt{4t \frac{\tanh 2\pi t|\la|}{2\pi t|\la|}},\lambda,\sqrt{4t \eta \frac{\tanh 2\pi t|\la|}{2\pi t|\la|}}\right) d\eta d\xi.
\end{align*}
Therefore, using also the identity $\int_{\R^m\times(0,\infty)} \eta^{-s}e^{-|\xi|^2-\eta}d\eta d\xi=\G(1-s)\pi^{\frac{m}2}$ and recalling that $g_0(\lambda)=\hat{\phi}(0,\lambda,0)$, we can write
\begin{align*}
&\int_{\R^k}\left|g_t(\lambda)-g_0(\lambda)\right|d\lambda=\frac{1}{\G(1-s)\pi^{\frac{m}2}}\int_{\R^k}\left|\left(\frac{1}{\cosh 2\pi t|\la|}\right)^{\frac m2+1-s} \int_{\R^m\times(0,\infty)} \eta^{-s}e^{-|\xi|^2-\eta}\right.\times\\
&\times\hat{\phi}\left(\xi\sqrt{4t \frac{\tanh 2\pi t|\la|}{2\pi t|\la|}},\lambda,\sqrt{4t \eta \frac{\tanh 2\pi t|\la|}{2\pi t|\la|}}\right) d\eta d\xi +\\
&-\left.g_0(\lambda)\int_{\R^m\times(0,\infty)} \eta^{-s}e^{-|\xi|^2-\eta}d\eta d\xi \left(\left(\frac{1}{\cosh 2\pi t|\la|}\right)^{\frac m2+1-s}+1-\left(\frac{1}{\cosh 2\pi t|\la|}\right)^{\frac m2+1-s}\right) \right|d\lambda\\
&\leq \frac{1}{\G(1-s)\pi^{\frac{m}2}}\int_{\R^k}\int_{\R^m\times(0,\infty)}\eta^{-s}e^{-|\xi|^2-\eta}\times\\
&\times\left| \hat{\phi}\left(\xi\sqrt{4t \frac{\tanh 2\pi t|\la|}{2\pi t|\la|}},\lambda,\sqrt{4t \eta \frac{\tanh 2\pi t|\la|}{2\pi t|\la|}}\right)-\hat{\phi}(0,\lambda,0)\right| d\eta d\xi d\lambda +\\
&+ \frac{1}{\G(1-s)\pi^{\frac{m}2}}\int_{\R^k}\int_{\R^m\times(0,\infty)}\eta^{-s}e^{-|\xi|^2-\eta}\hat{\phi}(0,\lambda,0)\left(1- \left(\frac{1}{\cosh 2\pi t|\la|}\right)^{\frac m2+1-s}\right)d\eta d\xi d\lambda.
\end{align*}
Since $\phi\in \mathscr{S}\left(\R^m\times\R^k\times [0,\infty)\right)$, we know that there exists $C_\phi>0$ such that
$$
\left|\hat{\phi}(z,\lambda,y)\right|\leq \frac{C_\phi}{1+|\lambda|^{k+1}}\quad\mbox{ for all }(z,\lambda,y)\in \R^m\times\R^k\times(0,\infty)
$$
and that moreover
$$
\hat{\phi}(z,\lambda,y)\mbox{ is continuous at }(z,y)=(0,0) \,\,\mbox{(the continuity is actually uniform in $\lambda\in\R^k$)}.
$$
The last two properties allow to exploit the dominated convergence theorem and infer that:
$$
\underset{t\to 0^+}{\lim} \int_{\R^k}\int_{\R^m\times(0,\infty)}\eta^{-s}e^{-|\xi|^2-\eta}\hat{\phi}(0,\lambda,0)\left(1- \left(\frac{1}{\cosh 2\pi t|\la|}\right)^{\frac m2+1-s}\right)d\eta d\xi d\lambda=0,
$$
and
\begin{align*}
\underset{t\to 0^+}{\lim}\int_{\R^k}&\int_{\R^m\times(0,\infty)}\eta^{-s}e^{-|\xi|^2-\eta}\times\\
&\times\left| \hat{\phi}\left(\xi\sqrt{4t \frac{\tanh 2\pi t|\la|}{2\pi t|\la|}},\lambda,\sqrt{4t \eta \frac{\tanh 2\pi t|\la|}{2\pi t|\la|}}\right)-\hat{\phi}(0,\lambda,0)\right| d\eta d\xi d\lambda=0.
\end{align*}
In turn, these limiting relations imply that
$$
\int_{\R^k}\left|g_t(\lambda)-g_0(\lambda)\right|d\lambda \underset{t\to 0^+}{\longrightarrow} 0.
$$
This concludes the proof of \eqref{mainclaim}. From \eqref{mainclaim} and the continuity of the Fourier transform from $L^1$ to $L^\infty$, we deduce that
\begin{equation}\label{claimply}
\check{g}_t \to \check{g}_0 \quad \mbox{ in }L^\infty\left(\R^k\right)\quad\mbox{ as }t\to 0^+,
\end{equation}
where we have used the notation $\check{g}$ for the inverse Fourier transform $\check{g}(\sigma)=\int_{\R^k}e^{2\pi i \left\langle \lambda, \sigma\right\rangle} g(\lambda) d\lambda$. Moreover, for any $\sigma_0\in\R^k$, we have $\check{g}_0(\sigma_0)=\phi(0,\sigma_0,0)$ and, by an application of Fubini's theorem, also
\begin{align*}
&\check{g}_t(\sigma_0)=\frac{2\pi^{1-s}}{\G(1-s)(4\pi t)^{\frac{m}2 +1-s}}\times\\
&\times\int_{\R^k} \int_{\R^m\times(0,\infty)} e^{2\pi i\left\langle \sigma_0,\lambda\right\rangle} \left(\frac{2\pi t|\la|}{\sinh 2\pi t|\la|}\right)^{\frac m2+1-s} e^{-\frac{|z|^2 +y^2}{4t}\frac{2\pi t|\la|}{\tanh 2\pi t|\la|}} \hat{\phi}(z,\lambda,y) y^{1-2s} dy dzd\lambda\\
& =\frac{1}{(4\pi t)^{\frac{m}2 +1-s}}\int_{\R^m\times\R^k\times(0,\infty)} \int_{\R^k} e^{-2\pi i\left\langle \sigma-\sigma_0,\lambda\right\rangle} \left(\frac{2\pi t|\la|}{\sinh 2\pi t|\la|}\right)^{\frac m2+1-s}\times\\
&\times e^{-\frac{|z|^2 +y^2}{4t}\frac{2\pi t|\la|}{\tanh 2\pi t|\la|}} \phi(z,\sigma,y) d\lambda d\mu \\
&=\int_{\R^m\times\R^k\times (0,\infty)} \mathfrak{q}_{(s)}((z,\sigma-\sigma_0),t,y) \phi(z,\sigma,y) d\mu.
\end{align*}
Hence, \eqref{claimply} implies that
$$
\int_{\R^m\times\R^k\times (0,\infty)} \mathfrak{q}_{(s)}((z,\sigma-\sigma_0),t,y) \phi(z,\sigma,y) d\mu \underset{t\to 0^+}{\longrightarrow}  \phi(0,\sigma_0,0)\quad\mbox{uniformly for }\sigma_0\in\R^k.
$$
In particular, when applied to $\sigma_0=0$, this completes the proof of \eqref{finally} and finishes the proof of Theorem \ref{T:ext0}.
}


\bibliographystyle{amsplain}

\begin{thebibliography}{10}

\bibitem{Ba}
S. M. Baouendi, \emph{Sur une classe d'op\'erateurs elliptiques d\'eg\'en\'er\'es}. Bull. Soc. Math. France \textbf{95}~(1967), 45-87.

\bibitem{BFIh}
W. Bauer, K. Furutani \& C. Iwasaki, \emph{Spectral analysis and geometry of sub-Laplacian and related Grushin-type operators}. Partial differential equations and spectral theory. Oper. Theory Adv. Appl., vol. 211, 183-290, Birkh\"{a}user/Springer Basel AG, Basel, 2011.

\bibitem{BFI}
W. Bauer, K. Furutani \& C. Iwasaki, \emph{Fundamental solution of a higher step Grushin type operator}. Adv. Math. \textbf{271}~(2015), 188-234.
  
\bibitem{BGG}
R. Beals, B. Gaveau \& P. Greiner, \emph{The Green function of model step two hypoelliptic operators and the analysis of certain tangential Cauchy Riemann complexes}. Adv. Math. \textbf{121}~(1996), no. 2, 288-345.

\bibitem{BGG2}
R. Beals, B. Gaveau \& P. Greiner, \emph{Green's functions for some highly degenerate elliptic operators}. J. Funct. Anal. \textbf{165}~(1999), no. 2, 407-429.

\bibitem{B}
R. Beals, \emph{A note on fundamental solutions}. Comm. Partial Differential Equations \textbf{24}~(1999), no. 1-2, 369-376.

\bibitem{BGV}
N. Berline, E. Getzler \& M. Vergne, \emph{Heat kernels and Dirac operators}. Grundlehren der Mathematischen Wissenschaften [Fundamental Principles of Mathematical Sciences], 298. Springer-Verlag, Berlin, 1992.

\bibitem{BR}
A. Boggess \& A. Raich, \emph{A simplified calculation for the fundamental solution to the heat equation on the Heisenberg group}. Proc. Amer. Math. Soc. \textbf{137}~(2009), no. 3, 937-944.

\bibitem{Bri}
H. C. Brinkman, \emph{Brownian motion in a field of force and the diffusion theory of chemical reactions. II}. Physica \textbf{23}~(1956), 149-155.

\bibitem{BGTms} F. Buseghin, N. Garofalo \& G. Tralli, {\it On the limiting behaviour of some nonlocal seminorms: a new phenomenon}. To appear in Ann. Sc. Norm. Super. Pisa Cl. Sci., DOI: \verb|10.2422/2036-2145.202005_014|.

\bibitem{CCFI}
O. Calin, D.-C. Chang, K. Furutani \& C. Iwasaki, \emph{Heat kernels for elliptic and sub-elliptic operators}. Applied and Numerical Harmonic Analysis, Birkh\"{a}user/Springer, New York, 2011.

\bibitem{CCGGL}
C.-H. Chang, D.-C. Chang, B. Gaveau, P. Greiner \& H.-P. Lee, \emph{Geometric analysis on a step 2 Grusin operator}. Bull. Inst. Math. Acad. Sin. (N.S.) \textbf{4}~(2009), no. 2, 119-188.

\bibitem{CT}
D.-C. Chang \& J. Tie, \emph{Estimates for powers of sub-Laplacian on the non-isotropic Heisenberg group}. J. Geom. Anal. \textbf{10}~(2000), no. 4, 653-678.

\bibitem{CDKR}
M. Cowling, A. H. Dooley, A. Kor\'anyi \& F. Ricci, \emph{$H$-type groups and Iwasawa decompositions}. Adv. Math. \textbf{87}~(1991), no. 1, 1-41.

\bibitem{Cy}
J. Cygan, \emph{Heat kernels for class $2$ nilpotent groups}. Studia Math. \textbf{64}~(1979), no. 3, 227-238. 

\bibitem{DFP}
M. Di Francesco \& A. Pascucci, \emph{On a class of degenerate parabolic equations of Kolmogorov type}. AMRX Appl. Math. Res. Express (2005), no. 3, 77-116. 

\bibitem{Fo}
G. B. Folland, \emph{Subelliptic estimates and function spaces on nilpotent Lie groups}. Ark. Mat. \textbf{13}~(1975), no. 2, 161-207. 

\bibitem{FGMT}
R. L. Frank, M. d. M. Gonz\'alez, D. Monticelli \& J. Tan, \emph{An extension problem for the $CR$ fractional Laplacian}. Adv. Math. \textbf{270}~(2015), 97-137. 

\bibitem{Gjde}
N. Garofalo, \emph{Unique continuation for a class of elliptic operators which degenerate on a manifold of arbitrary codimension}. J. Diff. Equations \textbf{104}~(1993), no. 1, 117-146. 

\bibitem{Gems}
N. Garofalo, \emph{Hypoelliptic operators and some aspects of analysis and geometry of sub-Riemannian spaces}. Geometry, analysis and dynamics on sub-Riemannian manifolds. Vol. 1, 123-257, EMS Ser. Lect. Math., Eur. Math. Soc., Z\"urich, 2016. 

\bibitem{GThls}
N. Garofalo \& G. Tralli, \emph{Hardy-Littlewood-Sobolev inequalities for a class of non-symmetric and non-doubling hypoelliptic semigroups}. To appear in Math. Ann., DOI: \verb|10.1007/s00208-020-02090-6|.

\bibitem{GTiso}
N. Garofalo \& G. Tralli, \emph{Nonlocal isoperimetric inequalities for Kolmogorov-Fokker-Planck operators}. J. Funct. Anal. \textbf{279}~(2020), 108591.

\bibitem{GTbbmd}
N. Garofalo \& G. Tralli, \emph{A Bourgain-Brezis-Mironescu-D\'avila theorem in Carnot groups of step two}. To appear in Comm. Anal. Geom. (ArXiv preprint 2004.08529).

\bibitem{GTfeel}
N. Garofalo \& G. Tralli, \emph{Feeling the heat in a group of Heisenberg type}. Adv. Math. \textbf{381}~(2021), 107635. 

\bibitem{GTinter}
N. Garofalo \& G. Tralli, \emph{A heat equation approach to intertwining}. To appear in J. Anal. Math. (ArXiv preprint 2011.10828).

\bibitem{Gav}
B. Gaveau, \emph{Principe de moindre action, propagation de la chaleur et estim\'ees sous elliptiques sur certains groupes nilpotents}. Acta Math. \textbf{139}~(1977), no. 1-2, 95-153. 

\bibitem{Gr1}
V. V. Gru$\check{s}$in, \emph{A certain class of hypoelliptic operators}. Mat. Sb. (N.S.) \textbf{83} (125)~(1970), 456-473.

\bibitem{Gr2}
V. V. Gru$\check{s}$in, \emph{A certain class of elliptic pseudodifferential operators that are degenerate on a submanifold}. Mat. Sb. (N.S.) \textbf{84} (126)~(1971), 163-195.

\bibitem{Ho}
L. H{\"o}rmander,
\textit{Hypoelliptic second order differential equations}. Acta Math. \textbf{119}~(1967), 147-171.

\bibitem{HJ}
R. A. Horn \& C. R. Johnson, \emph{Matrix analysis}. Cambridge University Press, Cambridge, corrected reprint of the 1985 original, 1990.

\bibitem{Hu}
A. Hulanicki, \emph{The distribution of energy in the Brownian motion in the Gaussian field and analytic-hypoellipticity of certain subelliptic operators on the Heisenberg group}. Studia Math. \textbf{56}~(1976), no. 2, 165-173. 

\bibitem{Ka}
A. Kaplan, \emph{Fundamental solutions for a class of hypoelliptic PDE generated by composition of quadratic forms}. Trans. Amer. Math. Soc. \textbf{258}~(1980), no. 1, 147-153.

\bibitem{Kli}
A. Klingler, \emph{New derivation of the Heisenberg kernel}. Comm. Partial Differential Equations \textbf{22}~(1997), no. 11-12, 2051-2060.

\bibitem{Kol}
A. N. Kolmogorov,  
\textit{Zuf\"allige Bewegungen (Zur Theorie der Brownschen Bewegung)}. Ann. of Math. (2) \textbf{35}~(1934), 116--117.

\bibitem{LaPo}
E. Lanconelli \& S. Polidoro, \emph{On a class of hypoelliptic evolution operators}. Partial differential equations, II (Turin, 1993). Rend. Sem. Mat. Univ. Politec. Torino \textbf{52}~(1994), no. 1, 29-63. 

\bibitem{LP}
F. Lust-Piquard, \emph{A simple-minded computation of heat kernels on Heisenberg groups}. Colloq. Math. \textbf{97}~(2003), no. 2, 233-249.

\bibitem{MM}
A. Martini \& D. M\"uller, \emph{Spectral multipliers on $2$-step groups: topological versus homogeneous dimension}. Geom. Funct. Anal. \textbf{26}~(2016), no. 2, 680-702.

\bibitem{MOZ}
J. M\"ollers, B. Orsted \& G. Zhang, \emph{On boundary value problems for some conformally invariant differential operators}. Comm. Partial Differential Equations \textbf{41}~(2016), no. 4, 609-643.

\bibitem{OU}
L. S. Ornstein \& G. E. Uhlenbeck, \emph{On the theory of the Brownian motion. I}.
Phys. Rev. (2) \textbf{36}~(1930), 823-841. 

\bibitem{Ran}
J. Randall,  \emph{The heat kernel for generalized Heisenberg groups}. J. Geom. Anal. \textbf{6}~(1996), no. 2, 287-316.

\bibitem{RTaim}
L. Roncal \& S. Thangavelu, \emph{Hardy's inequality for fractional powers of the sublaplacian on the Heisenberg group}. Adv. Math. \textbf{302}~(2016), 106-158.

\bibitem{RT}
L. Roncal \& S. Thangavelu, \emph{An extension problem and trace Hardy inequality for the sublaplacian on $H$-type groups}. Int. Math. Res. Not. IMRN (2020), no. 14, 4238-4294.


\end{thebibliography}

\end{document}